\documentclass[a4paper]{amsart}       
\usepackage[english]{babel}
\usepackage[T1]{fontenc}
\usepackage[latin9,ansinew]{inputenc}
\usepackage{amsmath,amsthm,amssymb}
\usepackage{enumerate,color}
\usepackage[active]{srcltx}
\usepackage{tikz}
\usepackage{hyperref}
\hypersetup{
    colorlinks,
    linkcolor={blue!50!black},
    citecolor={blue!50!black},
    urlcolor={blue!80!black}
}

\usepackage{mathtools}
\usepackage{layouts}
\usepackage{setspace,enumitem}
\usepackage{bbm}

\numberwithin{equation}{section}
\numberwithin{figure}{section}

\renewcommand{\a}{\alpha}

\renewcommand{\d}{\delta}
\newcommand{\g}{\gamma}
\newcommand{\e}{\varepsilon}
\newcommand{\f}{\varphi}

\newcommand{\s}{\sigma}
\renewcommand{\r}{\varrho}

\renewcommand{\t}{\tau}

\newcommand{\om}{\omega}
\newcommand{\D}{\Delta}

\newcommand{\Om}{\Omega}

\newcommand{\sets}[1]{\mathbb{#1}}
\newcommand{\CC}{\sets{C}}
\newcommand{\EE}{\sets{E}}
\newcommand{\NN}{\sets{N}}
\newcommand{\PP}{\sets{P}}
\newcommand{\QQ}{\sets{Q}}
\newcommand{\RR}{\sets{R}}

\newcommand{\nice}[1]{\mathcal{#1}}
\newcommand{\nB}{\nice{B}}
\newcommand{\nC}{\nice{C}}
\newcommand{\nD}{\nice{D}}
\newcommand{\nE}{\nice{E}}
\newcommand{\nF}{\nice{F}}

\newcommand{\nI}{\nice{I}}
\newcommand{\nK}{\nice{K}}
\newcommand{\nL}{\nice{L}}
\newcommand{\nM}{\nice{M}}

\newcommand{\nO}{\nice{O}}
\newcommand{\nS}{\nice{S}}

\newcommand{\nU}{\nice{U}}

\newcommand{\wt}[1]{\widetilde{#1}}
\newcommand{\wh}[1]{\widehat{#1}}

\newcommand{\arre}{\rightarrow}

\newcommand{\arrse}{\searrow}
\newcommand{\arrne}{\nearrow}

\newcommand{\dd}{\,\mathrm{d}}
\newcommand{\ee}{\mathrm{e}}
\newcommand{\ii}{\mathrm{i}}

\newcommand{\as}{\text{  a.s.}}
\newcommand{\id}{\mathrm{id}}
\newcommand{\cadlag}{c\`adl\`ag }

\newcommand{\conj}{\overline}
\newcommand{\<}{\langle}
\renewcommand{\>}{\rangle}

\newcommand{\cf}{\mathbbm{1}}
\newcommand{\ten}{\otimes}

\newtheoremstyle{flowstyle}{3pt}{3pt}{}{0pt}{\bf\sffamily}{. }{ }{}

\theoremstyle{flowstyle}
{
\newtheorem{theorem}{Theorem}
\newtheorem{corollary}[theorem]{Corollary}

\newtheorem{proposition}[theorem]{Proposition}
\newtheorem{lemma}[theorem]{Lemma}
\newtheorem{definition}[theorem]{Definition}
\newtheorem{example}[theorem]{Example}
\newtheorem{remark}[theorem]{Remark}
\newtheorem*{theorem*}{Theorem}
}
\numberwithin{theorem}{section}

\begin{document}

\title{L\'evy processes with values in locally convex Suslin spaces}

\author{Florian Baumgartner}


\maketitle

\begin{abstract}
We provide a L\'evy-It\^o decomposition of sample paths of L\'evy processes with values in complete locally convex Suslin spaces. This class of state spaces contains the well investigated examples of separable Banach spaces, as well as Fr\'echet or distribution spaces among many others. Sufficient conditions for the existence of a pathwise compensated Poisson integral handling infinite activity of the L\'evy process are given.
\end{abstract}

\section{Introduction}\label{sec:introduction}
The present paper is concerned with L\'evy processes with infinite-dimensional state spaces beyond Banach spaces. We assume the state space to be a \emph{complete locally convex Suslin space}. A fundemental result in the analysis of L\'evy processes is the decomposition of sample paths into independent diffusion and jump components. As a main result of this paper we obtain this so-called L\'evy-It\^o decomposition:
\begin{theorem*}[L\'evy-It\^o-decomposition]
Let $X$ be a L\'evy process in a locally convex Suslin space $E$ with characteristics $(\g,Q,\nu,K)$ and let the L\'evy measure $\nu$ be \emph{locally reducible} with reducing set $K$. Then there exist an $E$-valued Wiener process $(W_t)_{t\in T}$ with covariance operator $Q$, an independently scattered Poisson random measure $N$ on $T \times E$ with compensator $\lambda\ten \nu$ and a set $\Om_0\in \nF$ with $\PP(\Om_0)=1$ such that for all $\omega\in\Om_0$ one has
\begin{align}
X_t(\om)=\g t \;+ \;W_t(\om) \;+\!\! \int\limits_{[0,t]\times K}\!\!\!x\dd \wt{N}(s,x)(\om)\;+\!\!\int\limits_{[0,t]\times K^c}\!\!\! x\dd N(s,x)(\om)
\end{align}
for all $t\in T$. Furthermore, all the summands in \eqref{eq:levyitolcs} are independent and the convergence of the first integral in the sense of \eqref{eq:Jtdef} is a.s.\ uniform in $t$ on bounded intervals in $E_K$ and $E$.

Additionally, if the compact sets of $E$ admit a fundamental system of compact separable Banach disks, then there exists a separable compactly embedded Banach space $F\subseteq E$ such that the first three summands in the L\'evy-It\^o decomposition take values in $F$ a.s.
\end{theorem*}
A compact separable Banach disk $K$ is a compact set such that its linear hull is a separable Banach space with respect to the closed unit ball $K$. The crucial notion of local reducibility and the precise definition of all terms in the theorem will be given and investigated in this paper.

This theorem extends and unifies almost all known results of L\'evy processes with values in topological vector spaces. L\'evy processes with values in $\RR^d$ \cite{sato:99,applebaum:09} and in Banach spaces \cite{stochintlevyitoBtype2,levystochbanachappleb,stochintlevybanach,itoformulapoissonrmbanach,dettweiler} have been well-established but beyond the Banach space case little is known.

Nevertheless, much can be found regarding more general stochastic processes or special cases of L\'evy processes in topological vector spaces: Wiener processes with values in locally convex spaces \cite{feyel_delapradelle_95,bogachev} or Markov processes in completely regular Suslin spaces \cite{schwartz_markov} were considered; K.\ It\^o presents a theory of stochastic processes with values in distribution spaces in order to solve abstract Cauchy problems, cf. \cite{ito:84}.  In different settings, SPDEs with solution processes in nuclear and duals of nuclear spaces have been investigated by various authors, cf.\ \cite{kallianpurxiong,kallianpurxionglecture,stochintnuclear,kallperez_SEENuclearSpace,kallperez_weakconvergence,bojdecki_gorostiza_langevin,bojdecki_jakubowski_Girsanov,kumar_currfluc,goros_navarro_rodrigues_fluctuations}. \smallskip \\
Returning to L\'evy processes, \"Ust\"unel presented a L\'evy-It\^o decomposition in a class of Suslin nuclear duals of nuclear spaces \cite{addprocessesnuclear}. However, the decomposition had some shortcomings, as missing independence of components and an $L^2$-converging integral not allowing to claim a desireable pathwise decomposition. Recently, C.\ Fonseca Mora was able to both enlarge the class of state spaces by new methods, cf.\ \cite{fonsecamora:15} and prove a satisfactory L\'evy-It\^o decomposition for duals of reflexive nuclear spaces.\smallskip\\
In this paper, a different approach based on the works of Dettweiler and Tortrat on infinitely divisible Radon measures on locally convex spaces, cf.\ \cite{dettweilerbadrikianisch,tortrat69}, allows to drop any nuclearity assumptions on the state space $E$. This extends and unifies the results of \"Ust\"unel (in the nuclear setting) and Dettweiler \cite{dettweiler} (for Banach spaces).

For example, this approach allows to treat a L\'evy process in the locally convex direct sum $L^p([0,1])\oplus \nD(\RR^n)$ which is obviously neither a Banach nor a nuclear space but a locally convex Suslin. Also, projecting on the components does not work in general as they might not be independent (as in the finite-dimensional case). So, even in this simple example, a unified approach is necessary. 

The main technique is the reduction of the small jumps to a compactly embedded separable Banach subspace $E_K$ of $E$ -- if possible, i.e., if the L\'evy measure $\nu$ is \emph{locally reducible}, cf.\ Definition~\ref{def:locallyreducible}. This will guarantee a.s.\ uniform convergence of a Poisson integral representing the small jumps of a L\'evy process:

\begin{theorem*}
Let $\nu$ be locally reducible with reducing set $K$. For $t\in T$, the quantity
\begin{align}
J_t:=\int_{(0,t]\times K} x\dd \wt{N}(s,x):=\sum_{n=1}^\infty  J([0,t]\times C_n)
\end{align}
is a series of independent random variables in $E_K=\bigcup_{n\in\NN} n K$ and converges almost surely in $E_K$ and $E$. The convergence is uniform in $t$ on bounded intervals of $T$. Finally, $(J_t)_{t\in T}$ is a \cadlag L\'evy process in $E$ with characteristics $(0,0,\nu|_{K},K)$.
\end{theorem*}
In Section~\ref{sec:separablebanachspaces} we investigate this reduction technique and characterise locally reducible L\'evy measures (Theorem~\ref{thm:restLevy}). It should be pointed out that this result is even a little better than in \cite[Proof of Theorem~2.1]{dettweiler}, as the series is converging a.s.\ uniformly without any subsequence arguments.

On the other hand, a zero-one law for generalised Poisson exponentials by Janssen \cite{janssen_01} will allow us to impose a simple and natural condition to the state space obtaining the following useful result:
\begin{theorem*}
If the compact sets of $E$ admit a fundamental system of compact separable Banach disks, then every L\'evy measure on $E$ is locally reducible.
\end{theorem*}

This property is simpler to check and satisfied e.g.\ by separable Fr\'echet and all common distribution spaces.\smallskip\\
Another peculiarity happens taking into account the possibly uncountable neighbourhood bases of $E$. Measurability problems in connection with limits can be overcome introducing a weak metric on $E$ and exploiting the Suslin property of $E$ which is essential in this approach, cf.\ Section~\ref{sec:cadlagfunctionsinsuslinspaces}. \medskip\\
The paper is organised as follows. After some preliminaries on spaces and measures, vector valued L\'evy processes will be defined in Section~\ref{sec:levyprocesses}. Section~\ref{sec:separablebanachspaces} is dedicated to the reduction of the small jumps to the Banach subspace, Section~\ref{sec:cadlagfunctionsinsuslinspaces} deals with \cadlag functions with Suslin space values, Section~\ref{sec:brownianmotion} with Wiener processes. The main results are Theorem~\ref{thm:comppoisint} in Section~\ref{sec:jumpprocesses} and the L\'evy-It\^o decomposition, Theorem~\ref{thm:levyito}, in Section~\ref{sec:levyito}.

\section{Preliminaries}\label{sec:preliminaries}
\subsection{Spaces.}
Throughout this paper, we will make the following assumption on the state space $E$ (unless explicitly stated differently):
\begin{enumerate}[label=(S\arabic*)]
\item $E$ is a real locally convex space,
\item $E$ complete, and
\item $E$ is a Suslin space.
\end{enumerate}
A complete locally convex space has the property that every Cauchy net has a limit. A \emph{Suslin space} is a Hausdorff (topological) space which is a surjective continuous image of a polish space. We denote by $E'$ the topological dual of $E$ and $\nB(E)$ is the Borel-$\s$-algebra of $E$. The cylindrical $\s$-algebra $\nE(E)$ is generated by the elements of $E'$. The above assumptions guarantee the following frequently used properties: $E$ is separable, Hausdorff and completely regular (i.e.\ a point and a closed set can be separated by a continuous function), $\nE(E)=\nB(E)$ and $(E,\nB(E))$ is a measurable vector space, i.e., addition and scalar multiplication are measurable. Furthermore, there exists a sequence of elements in $E'$ separating the points of $E$ and therefore one has a continuous metric on $E$. 

The following essential result is due to L.~Schwartz:
\begin{proposition}[Corollary 2,~p.~101 of {\cite{schwartz_radon}}]\label{prop:comparablesuslin}
Let $\t_1$ and $\t_2$ two comparable Suslin topologies on a set $F$. Then, the respective Borel-$\s$-algebras coincide.
\end{proposition}

\subsection{Measures.}
The set of (nonnegative) Borel measures on $E$ is denoted by $\nM(E)$, its subset of bounded, resp.\ probability measures by $\nM^b(E)$, $\nM^1(E)$, respectively. A measure $\mu\in\nM^b(E)$ is called Radon, if for $B\in\nB(E)$ and $\e>0$ there exists a compact set $K\subseteq B$ such that $\mu(B\setminus K)<\e$.
Every finite Borel measure on any Suslin space is Radon \cite[Thm. 10, p. 122]{schwartz_radon}. The Fourier transform of a measure $\mu\in\nM^b(E)$ is given by
$$\wh{\mu}\colon\,E'\rightarrow \CC\qquad \wh{\mu}(a):=\int_E \ee^{\ii\<x,a\>}\dd \mu(x).$$
Any two finite measures with equal Fourier transform coincide, cf.\ \cite[Theorem~2.2,~p.~200]{vakhania_tarieladze}. For $\mu,\nu\in\nM(E)$ the convolution is defined by $$\mu*\nu(B)=(\mu\ten\nu)_\a(B)=\mu\ten\nu(\a^{-1}(B)),$$
where $\a: E\times E\arre E,\,(x,y)\mapsto x+y$, cf.\ \cite[Appendix,~p.\ 373f]{bogachev}. \sloppypar
For $A\in\nB(E)$, the restriction of a measure $\mu\in\nM(E)$ to $A$ is defined by $\mu|_A(B):=\mu(B\cap A)$, $B\in\nB(E)$. Note that $\mu|_A\in\nM(E)$. If $\mu|_A$ is only considered on $\nB(A)$, we denote it by $\mu\|_A\in\nM(A)$.
\subsection{Infinitely divisible measures.}
A measure $\mu\in\nM^1(E)$ is called \emph{Poissonian} provided that there exists $\nu\in\nM^b(E)$  such that its Fourier transform satisfies
$$\wh{\mu}(a)=\ee^{\int_E \left(\ee^{\ii\<x,a\>}-1\right)\dd \nu(x)}.$$
For a measure $\nu\in\nM^b(E)$ its \emph{Poisson exponential} can be defined by
\begin{align}\label{eq:PoissonSeries}
\ee(\nu):=\ee^{-\nu(E)}\sum_{n=0}^\infty \frac{\nu^{*n}}{n!},
\end{align}
with a setwise converging series. In this case, $\ee(\nu)$ is indeed a Poisson\-ian measure with associated measure $\nu$, cf.\  \cite[p.~288]{dettweilerbadrikianisch}.
A measure $\r\in\nM^1(E)$ is called \emph{Gaussian} if for any $a\in E'$ the measure $\r\circ a^{-1}$ is Gaussian. A measure $\mu\in \nM^1(E)$ is called \emph{infinitely divisible} provided that for every $n\in \NN$ there exists a measure $\mu_n\in\nM^1(E)$ such that $\mu=\mu_n^{*n}$. Gaussian and Poisson measures are infinitely divisible. The set $\nI(E)$ of infinitely divisible measures on $E$ is closed in $\nM^1(E)$  in the weak topology, cf.\ \cite[Korollar~1.10]{dettweilerbadrikianisch}.

A set $M\subseteq \nM^b(E)$ of finite measures on $E$ is \emph{uniformly tight} if $\sup_{\mu \in M} \mu(E) <\infty$ and if for all $\e>0$ there exists a compact set $K\subseteq E$ such that $\mu(K^c)<\e$ for all $\mu\in M$.
The set $M$ is called \emph{shift tight}, if for every $\mu\in M$ there exists an $x_\mu\in E$ such that the familiy $(\mu*\d_{x_\mu})_{\mu\in M}$ is uniformly tight.\sloppypar
The following definition is from \cite{dettweilerbadrikianisch,tortrat69}.
\begin{definition}
A measure $\nu\in\nM(E)$ is called \emph{L\'evy measure} if it satisfies
\begin{enumerate}
\item $\nu(\{0\})=0$;
\item  there exists an upwards directed set of finite measures (w.r.t.\ $\leq$) $\{\mu_i\in\nM^b(E) : i\in I\}$ with $\mu_i\leq \nu$ and $\sup_i \mu_i=  \nu$ (setwise) and such that the family of Poisson measures $(\ee(\mu_i))_{i\in I}$ is shift tight.
\end{enumerate}
\end{definition}
If $E$ is a separable Banach space, this definition coincides with the following characterisation of a L\'evy measure: $\nu(\{0\})=0$, $\nu(B_\d^c)<\infty$ for all $\d>0$ and where $B_\d$ is the ball with radius $\d$, and for each positive sequence $\d_n\arrse 0$, the set $\{\ee(\nu|_{B_{\d_n}^c}),n\in\NN\}$ is shift tight, cf.\ \cite[Theorem~3.4.9]{heyer}. The above definition is thus indeed an extention of well-known concepts.\sloppypar
There exists a \emph{L\'evy-Khintchine-decomposition} of infinitely divisible measures:
\begin{theorem}[Dettweiler, {\cite[Satz~2.5]{dettweilerbadrikianisch}}]\label{thm:levy_khintchine} For $\mu\in \nI(E)$ there exist $\g\in E$, a linear symmetric and positive definite operator $Q\colon E'\arre E$, an absolutely convex and compact set $K\subseteq E$ and a L\'evy measure $\nu$ such that the characteristic function of $\mu$ has the form
\begin{align*}
\wh{\mu}(a) = \exp \left(\ii \<\g,a\>-\frac{1}{2}\<Qa,a\>+\!\int_E\big(\ee^{\ii\<x,a\>}\!-1-\ii\<x,a\>\cf_K(x)\big)\!\dd \nu(x)\right)
\end{align*}
for every $a\in E'$. $\nu$ and $Q$ are uniquely determined by $\mu$, $\g$ is unique after the choice of $K$.\\
Conversely, every measure $\mu\in\nM^1(E)$ with a Fourier transform of this type is infinitely divisible.
\end{theorem}
One says that $\mu$ has \emph{characteristics} $(\g,Q,\nu,K)$ if $\mu$ admits the above decomposition. The covariance operator is symmetric in the sense that $\<Qa,b\>=\<Qb,a\>$ for $a,b\in E'$.

\begin{lemma}[Dettweiler \cite{dettweilerbadrikianisch}, Lemma 1.5/Proof of Satz~2.5] \label{lem:cpK}
Let $\mu\in \nI(E)$. Then there exists a unique L\'evy measure $\nu$ associated to $\mu$ and some absolutely convex and compact set $K$ such that $\nu(K^c)<\infty$.
\end{lemma}

\begin{definition}
Let $\nu$ be a L\'evy measure and $x_\mu\in E$ be chosen in a way such that the family $(\ee(\mu) * \delta_{x_{\mu}})_{\mu\leq \nu}$  is uniformly tight. The \emph{generalised Poisson exponential} $\wt{\ee}(\nu)$ is an accumulation point of this family.
\end{definition}
The generalised Poisson exponential is unique up to translations \cite[p.\ 288]{dettweilerbadrikianisch}. We will sometimes use $\wt{\ee}(\nu)$ in relations like
$$\wt{\ee}(\nu) = \wt{\ee}(\nu|_K)*\ee(\nu|_{K^c})$$
meaning that choosing certain representatives of the generalised Poisson exponential, the equality holds up to a convolution with a Dirac measure on one side.

By Prokhorov's theorem, cf.\ \cite[Theorem~I.3.6]{vakhania_tarieladze}, the mentioned family is weakly relatively compact which justifies the definition as an accumulation point. We will use the fact that for a L\'evy measure, its generalised Poisson exponential $\wt{\ee}(\nu)$ is infinitely divisible. This follows from $\nI(E)$ being closed in $\nM^1(E)$.

\begin{corollary}\label{cor:levymeas}
For a measure $\nu\in\nM(E)$ the following assertions are equivalent:
\begin{enumerate}
\item $\nu$ is a L\'evy measure on $E$.
\item There exists a measure $\e\in\nM^1(E)$ such that its Fourier transform equals
$$\int_E\ee^{\ii\<x,a\>}\dd\e(x) = \exp \left( \int_E \big(\ee^{\ii \<x,a\>} -1 - \ii \<x,a\>\cf_K(x)\big)\dd \nu(x)\right),\quad a\in E'$$
for some absolutely convex and compact set $K\subseteq E$.
\end{enumerate}
\end{corollary}
\begin{proof}
(2) $\Longrightarrow$ (1): If $\wh{\e}$ has the above form, $\e\in\nI(E)$ and $\nu$ is the unique L\'evy measure corresponding to $\e$ by Theorem~\ref{thm:levy_khintchine}.\\
(1) $\Longrightarrow$ (2): If $\nu$ is a L\'evy measure, it follows from the proof of the L\'evy-Khintchine decomposition that there exists a generalised Poisson exponential $\e$ with the given Fourier transform (end of point 1 in the proof of Satz 2.5, \cite{dettweilerbadrikianisch}, where this measure is called $\nu_0*\ee(F_0)$).
\end{proof}

\section{L\'evy processes in locally convex spaces}\label{sec:levyprocesses}
An $E$-valued \emph{random vector} $X$ on a probability space $(\Om,\nF,\PP)$ is a measurable map $X\colon (\Om,\nF)\rightarrow (E,\nB(E)).$ An $E$-valued stochastic process $(X_t)_{t\in T}$ is a collection of $E$-valued random vectors over the parameter space $T=[0,t_{\max} ]$ with $t_{\max}>0$ or $T=[0,\infty)$. If there is no risk of confusion, we will omit the emphasis that a process is $E$-valued and call it simply a stochastic process.

A family $(\mu_t)_{t\in T}$ of probability measures on $(E,\nB(E))$ is called a \emph{convolution semigroup}, if $\mu_{t+s}=\mu_t*\mu_s$ for all $s,t,s+t\in T$ and $\mu_0=\d_0$. It is \emph{weakly continuous}, if $\mu_t$ converges to $\delta_0$ for $t\arrse 0$ in the weak topology of measures.

\begin{definition}[L\'evy processes]
An $E$-valued stochastic process $(X_t)_{t\in T}$ on a probability space $(\Om,\nF,\PP)$ with distributions $\mu_t:=\PP_{X_t}$ is a \emph{L\'evy process}, if
\begin{enumerate}[label=(L\arabic*)]
\item $X_t-X_s$ is independent of $\nF_s:=\s(X_r: r\leq s)$ for any $0\leq s<t$;
\item the distributions of $X_{t+s}-X_t$ and $X_s$ are equal for all $t$;
\item $X_0=0$ a.s.; and
\item \label{ax:wcsg} the family $(\mu_t)_{t\in T}$ is a weakly continuous convolution semigroup.
\end{enumerate}
\end{definition}
If $E=\RR^d$, the above definition yields the notion of a \emph{L\'evy process in law}, cf.\ \cite[Definition~1.6]{sato:99}, where \ref{ax:wcsg} is substituted by the equivalent property (cf.\ \cite[Proposition~1.4.1]{applebaum:09}) of stochastic continuity.

\begin{proposition}\label{prop:semigroup}
Let $\mu\in\nI(E)$.
\begin{enumerate}
\item The $n$-th root $\mu_{1/n}\in \nM^1(E)$ is unique.
\item Let $\eta\colon E'\arre \CC$ be the characteristic exponent of $\mu$ in Theorem~\ref{thm:levy_khintchine}. Then, $\wh{\mu_t}(a)=\ee^{t\eta(a)}$, $a\in E'$, $t\in T$.
\item There is a unique convolution semigroup $\mu_t:=\mu^{*t}$ embedded into $\mu$ such that $\mu_1=\mu$.

\end{enumerate}
\end{proposition}
\begin{proof}
(1):  On a locally convex Suslin space, $\nE(E)=\nB(E)$, therefore $\mu$ is uniquely determined on the generating $\pi$-system of cylindrical sets of the form $C=(a_1^{-1},\ldots,a_d^{-1})(B)$ with $B\in\nB(\RR^d)$ and  $a_1,\ldots,a_d\in E'$. So, let $\r_1$ and $\r_2$ be two different $n$-th roots of $\mu$. Then, there exist $a_1,\ldots,a_d\in E'$ such that $\r_1\circ (a_1,\ldots,a_d)^{-1}\neq \r_2\circ (a_1,\ldots,a_d)^{-1}$. These are two different $n$-th roots of $\mu\circ (a_1,\ldots,a_d)^{-1}\in \nI(\RR^d)$, a contradiction.\\
(2) and (3): From (1) it follows that $\mu_q:=\mu^{*q}$ is unique and definable for all rational $q$ by setting $\mu^{*q_1/q_2}:=\mu_{1/q_2}^{*q_1}$.
For $a\in E'$, the measure $\mu^{*q}_a:=\mu^{*q}\circ a^{-1}$ is infinitely divisible, and from the one-dimensional case it follows that its Fourier transform is $\f_a^q(u)=\ee^{q \eta_a(u)}=\ee^{q\eta(ua)}$, where $\eta\colon E'\arre \CC$ is the characteristic exponent of $\mu$, i.e.\ $\wh{\mu}(a)=\ee^{\eta(a)}$. This follows from $\nu_a$ being a L\'evy measure on $\RR$ (cf.\ \cite[Theorem~3.4.9]{heyer}).

For $q\arrse t$, the Fourier transforms $\f_a^q(u)\arre \f_a^t(u)$ for all $u\in \RR$, which yields (2). In other words, $\wh{\mu^{*q}}(a)\arre \wh{\mu^{*t}}(a)$ for all $a\in E'$. The family $(\mu^{*q})_{q\in\QQ,\,t\leq q\leq t_0}$ is uniformly tight by \cite[Satz~6.4]{Siebert} and  therefore weakly relatively compact by Prokhorov's theorem cf.\ \cite[Theorem~I.3.6]{vakhania_tarieladze}. This is sufficient to apply \cite[Theorem~IV.3.1]{vakhania_tarieladze} and obtain weak convergence of the net $\mu^{*q}, q\in \QQ,q\geq t$ for  $q\arrse t$.
\end{proof}
The previous proposition allows to keep the characteristics of a L\'evy process over time. The following existence theorem is proved in \cite{baumgartner:phd:15} and relies on the results of \cite{schwartz_markov}.
\begin{proposition}\label{prop:existence_levy_process1}
Let $(\mu_t)_{t\in T}$ be a weakly continuous convolution semigroup on $E$. Then, there exists a L\'evy process $(X_t)_{t \in T}$ on some probability space $(\Omega,\nF,\PP)$ with values in $E$ and such that $\PP_{X_t}= \mu_t$ for all $t\in T$.
Furthermore, one can choose a probability space such that the set of c\'adl\'ag paths has probability one.
\end{proposition}
From now on, we assume that an $E$-valued L\'evy process $X=(X_t)_{t\in T}$ is always given on a complete probability space $(\Omega,\nF,\PP)$ such that almost all paths of $X$ are \cadlag.\medskip\\
Let $\Om_0:=\{\om\in \Om\colon t\mapsto X_t(\om)\text{ is \cadlag}\!\!\}$. For $\om\in\Om_0$ we define
$$X_{t-}(\om):=\lim_{s\nearrow t} X_s(\om)\quad \text{and} \quad \D X_t(\om):=X_t(\om)-X_{t-}(\om).$$
For $\om\in \Om_0^c$, one sets $\Delta X_t(\om):=0$ for all $t\in T$.

\section{Separable Banach subspaces}\label{sec:separablebanachspaces}
In this section, suitable conditions on the L\'evy measure, allowing a reduction of the small jumps part to a Banach subspace, will be investigated.  \smallskip\\
For the functional analytic background of the following cf.\ \cite{jarchow}. A disk in a locally convex space is a bounded and absolutely convex set. We define the sets $\nB_0(E,\t)$ resp.\ $\nK_0(E,\t)$ of closed and compact disks, respectively. If there is no risk of confusion we omit the dependency on $E$ or its topology. For $B\in \nB_0$ (and therefore $\nK_0$) the linear hull
$$E_B:=\bigcup_{n\in\NN} n\cdot B$$
is a Banach space with respect to the gauge function $\|x\|_B:=\inf \{\rho>0 \colon x\in \rho\cdot B\}$, $x\in E_B.$
By boundedness of $B$ in $E$,  the canonical injection $\imath\colon E_B\hookrightarrow E$ is continuous (for $B\in\nK_0$ even a compact mapping).

\subsection{Measures on different underlying topologies}
We give some lemmas on certain types of measures with respect to different topologies.
\begin{lemma}\label{lem:gaussianmeasurecomparable}
Let $\t$ denote the given topology on $E$ and let $\t'$ be another comparable locally convex Suslin topology on $E$.  $\t$ and $\t'$ need not be complete. A measure $\r$ is a Gaussian measure with respect to $(E,\t)$ if and only if it is a Gaussian measure with respect to $\t'$.
\end{lemma}
\begin{proof}

Let $\r$ be Gaussian on $\nB(E_\t)=\nE(E,E'_\t)$. By \cite[Proposition~2.2.10]{bogachev}, this is equivalent to the statement that for the map $\psi\colon X\times X\arre X$, defined by $(x,y)\mapsto x \sin \f + y \cos \f$ it holds that $(\r\ten \r)(\psi^{-1}(B))=\r(B)$ for all $B\in \nB(E_{\t})=\nB(E_{\t'})=\nE(E,E'_{\t'})$. This means, $\r$ is Gaussian on $E_{\t'}$.
\end{proof}
\begin{lemma}\label{lem:poissonmeasurecomparable}
Let $\t$ denote the given topology on $E$ and let $\t'$ be another comparable locally convex Suslin topology on $E$.  $\t$ and $\t'$ need not be complete. A measure $\mu\in\nM^1(E)$ is a Poisson measure with respect to $\t$ if and only if it is a  Poisson measure with respect to $\t'$.
\end{lemma}
\begin{proof}
As recalled above, the Borel-$\s$-algebras of both topological spaces coincide and one obtains the result by considering the setwise converging series \eqref{eq:PoissonSeries} which only depends on the Borel structure.
\end{proof}

\begin{lemma}\label{lem:levymeasurecomparable}
Let $\t$ denote the given topology on $E$ and let $\t'$ be another comparable locally convex Suslin topology on $E$.  A measure $\nu$ is a L\'evy measure with respect to $(E,\t)$ if and only if it is a L\'evy measure with respect to $\t'$.
\end{lemma}
\begin{proof}
Without loss of generality assume $\t'\supseteq \t$. If $\nu$ is a L\'evy measure on $(E,\t')$, it follows from shift tightness of the defining family of the Poisson measures in $(E,\t')$ and thus $(E,\t)$ that it is a L\'evy measure on $(E,\t)$. \sloppypar
Conversely, if $\nu\in \nM(E)$ is a L\'evy measure with respect to $\t$, then there exists an infinitely divisible measure $\mu := \wt{\ee}(\nu)\in\nM^1(E)$. Now we consider $\t'$ as underlying topology. Note that infinite divisibility depends only on the Borel structure. Because $\mu$ is infinitely divisible, the L\'evy-Khintchine decomposition (w.r.t. $\t'$) (Theorem~\ref{thm:levy_khintchine}) provides a unique L\'evy measure $\nu'$. It follows from the converse implication that $\nu'=\nu$.
\end{proof}
\subsection{Restriction to a subspace.}
We begin with some lemmas, introduce the notion of local reducibility and a fundamental system of compact separable Banach disks and present the main result of this section, Theorem~\ref{thm:restLevy}.

\begin{lemma}\label{lem:EKprops}
\begin{enumerate}
\item Let $K$ be a Borel-measurable disk. The subspace $E_K$ is a Suslin space with the induced topology.
\item Let $K$ be a closed disk such that $(E_K,\|\cdot\|_K)$ is a separable Banach space. Then,
$$\nE(E_K)=\nB(E_K) = \nB(E)\cap E_K= \nE(E)\cap E_K$$
where $\nB(E_K)$ is induced by the norm topology in $E_K$ and $\nE(E_K):=\nE(E_K,{E_K}')$, where ${E_K}'$ denotes the linear and continuous functionals on $E_K$ with respect to the norm topology
\end{enumerate}
\end{lemma}
\begin{proof}
{(1)} follows from $E_K$ being Borel and every Borel subset of a Suslin space is Suslin \cite[Thm.~3,~p.~96]{schwartz_radon}.\sloppypar
{(2)}: $E_K$ is a Suslin space due to its separability, thus $\nB(E_K)=\nE(E_K)$, the same is true for the induced topology on $E_K$ which is Suslin by (1). Finally, two comparable Suslin topologies have the same Borel sets by Proposition~\ref{prop:comparablesuslin}.
\end{proof}

\begin{lemma}\label{lem:Lshift}
Let $M\subseteq \nM^1(E)$ be a family of measures with the following properties:
\begin{enumerate}
\item There is a measurable linear subspace $L\subseteq E$ such that $\mu(L)=1$ for all $\mu\in M$.
\item $M$ is shift tight, i.e.\ for $\e>0$ there is a compact set $K\subseteq E$ such that for all $\mu \in M$ there exist $x_\mu\in E$ with
$$\mu*\d_{x_\mu}(K)>1-\e.$$
\end{enumerate}
Then, there exist $x_\mu^0\in L$ such that
$$\mu* \d_{x_\mu^0}( 2 K^{\circ \circ})>\mu*\d_{x_\mu^0}(2(K^{\circ\circ})\cap L)>1-\e,$$
where $K^{\circ\circ}$ denotes the bipolar of $K$ (the the closed absolutely convex hull of $K$).
\end{lemma}
\begin{proof}
Without loss of generality we can assume that $K$ is absolutely convex as the absolutely convex hull of a compact set in a complete locally convex space is compact, cf.\ \cite[Proposition~6.7.2]{jarchow}. We have that $\mu(L)=1$ for every $\mu\in M$ and therefore, $\mu*\d_{x_\mu}(L-x_\mu)=1$. We define
$$K_\mu:=(L-x_\mu)\cap K,\quad \text{and}\quad \conj{K_\mu}:=-K_\mu=(L+x_\mu)\cap K$$
and obtain $\mu*\d_{x_\mu}(K_\mu)=\mu*\d_{x_\mu}(K)>1-\e$. We show that there exists a $\conj{u}\in\conj{K_\mu}$ with $x_\mu-\conj{u}\in L$. This follows from $1-\e<\mu*\d_{x_\mu}(K_\mu) = \mu(x_\mu+K_\mu)=\mu((x_\mu+K_\mu)\cap L)$. 
So, for every $\mu$, one can find such an element, we call it $\conj{u}_\mu\in \conj{K_\mu}$ such that $x_\mu^0:=x_\mu-\conj{u}_\mu\in L$. The following properties hold for every $\mu\in M$:
\begin{enumerate}
\item $\mu* \d_{x_\mu-\conj{u}_\mu}(K_\mu + \conj{u}_\mu)>1-\e$.
\item $K_\mu+\conj{u}_\mu \subseteq 2K=K+K$.
\item $K_\mu+\conj{u}_\mu \subseteq L$, because $\conj{u}_\mu+K_\mu\subseteq \conj{u}_\mu-x_\mu +L=L$ by definition of $K_\mu$.
\end{enumerate}
In particular, $\mu*\d_{x_\mu^0}(2K\cap L)\geq \mu*\d_{x_\mu^0}(K_\mu+\conj{u}_\mu)>1-\e$ for all $\mu\in M$.
\end{proof}

\begin{lemma}\label{lem:infinitedivsubspace}
Let $\mu\in\nI(E)$ and $L$ a linear subspace of $E$ with $\mu(L)=1$. Then, $\mu_{1/n}(L)=1$. In particular, $\mu\|_L$ is infinitely divisible on $L$.
\end{lemma}
\begin{proof}
Set $\r=\mu_{1/n}$ and assume a set $C\in\nB(E)$ with $C\cap L=\emptyset$ and $\r(C)>0$. Then, $\r^{*j}(C+\ldots +C)\geq \r(C)^j>0$ for every $j\in\NN$. We set $C_1:=C$.
If $\r^{*j}((C_{j-1}+C)\cap L)=0$ define $C_j:=(C_{j-1}+C)\setminus L$ and carry on by induction. If $j=n$ we obtain a contradiction. 

If there is a $k\in \{1,\ldots,n\}$ such that $\r^{*k}((C_{k-1}+C)\cap L)>0$ we have that also $\r^{*k}(L)>0$. 

Now, let $n=kl+m$ for $m<k$ and some integer $l$. By construction, $C_m\cap L = \emptyset$ and $\r^{*m}(C_m)>0$ holds, in particular, $(C_m+L)\cap L=\emptyset$. We obtain
$$\mu(C_m+L)=\r^{*m}*(\r^{*k})^{*l}(C_m+L)\geq \r^{*m}(C_m)\big(\r^{*k}(L)\big)^l>0,$$
a contradiction.
\end{proof}
\begin{definition}\label{def:locallyreducible}
\begin{enumerate}
\item We denote by
\begin{align*}
\nB_0^s:=\nB_0^s(E)&:=\{B\in\nB_0(E): E_B \text{ separable with respect to } \|\cdot\|_B \},
\end{align*}
and $\nK_0^s:=\nK_0^s(E):=\nB_0^s(E)\cap \nK_0(E)$.
The elements of $\nB_0^s$ ($\nK_0^s$) are called \emph{(compact) separable Banach disks}.
\item If for every $K\in \nK_0(E)$ there exists $B\in\nB_0^s(E)$ (resp.\ $\nK_0^s(E)$ such that $K\subseteq B$, the system $\nB_0^s(E)$ (resp.\ $\nK_0^s(E)$ is said to be \emph{fundemental} (for $\nK_0(E)$). In this case, $E$ is said to possess a fundamental family of separable Banach disks (resp.\ compact separable Banach disks).
\item A L\'evy measure $\nu\in\nM(E)$ is called \emph{locally reducible} if there exists a compact set $K\in\nK_0^s$ with $\nu(K^c)<\infty$ such that there exists a generalised Poisson exponential of $\nu|_K$ with $\wt{\ee}(\nu|_K)(E_K)=1$. In this case, $K$ is called $\nu$-\emph{reducing}. An infinitely divisible distribution is \emph{locally reducible} if its corresponding L\'evy measure has this property.
\end{enumerate}
\end{definition}
In Appendix~\ref{app:sfa} we give sufficient conditions and examples for spaces with a fundamental system $\nB_0^s(E)$. Concerning local reducibility, we immediately get:
\begin{lemma}\label{lem:immediate}
If $\nu$ is locally reducible and $K\in\nK_0^s$ is $\nu$-reducing, then every $H\in \nK_0^s$ with $H\supseteq K$ is $\nu$-reducing.
\end{lemma}

\begin{theorem}\label{thm:restLevy}
Let $\mu\in\nI(E)$ and $\nu\in\nM(E)$ its corresponding L\'evy measure. The following assertions are equivalent.
\begin{enumerate}
\item\label{it:mulc} $\mu$ is locally reducible.
\item\label{it:nulc} $\wt{\ee}(\nu|_K)(E_K)=1$ for some $K\in\nK_0^s(E)$ with $\nu(K^c)<\infty$, i.e.\ $\nu$ is locally reducible.
\item\label{it:fourier} The function
$$\f(a) = \exp\left(\int_{E_K} \ee^{\ii \<x,a\>}-1-\ii \<x,a\>\cf_K(x)\dd \nu(x)\right),\quad a\in E'$$
is a Fourier transform of some probability measure on $\nB(E_K)$ for some $K\in \nK_0^s(E)$ with $\nu(K^c)<\infty$.
\item\label{it:nurestB} The restriction of $\nu|_K$ to the subspace $(E_K,\|\cdot\|_K)$ is a L\'evy measure for some $K\in\nK_0^s(E)$ with $\nu(K^c)<\infty$.
\item\label{it:nurest} The restriction of $\nu|_K$ to the subspace $(E_K,\t)$ is a L\'evy measure for some $K\in\nK_0^s(E)$ with $\nu(K^c)<\infty$.
\item\label{it:lmdef} For some $K\in \nK_0^s(E)$ with $\nu(K^c)<\infty$ the following holds: There are $x_{\r}\in E$ such that for every $\e>0$ there exists an $n\in \NN$ such that for all $\r\leq \nu|_K$ one has
$$\ee(\r)*\d_{x_\r}(n\cdot K)>1-\e.$$
\end{enumerate}
Furthermore, if there is $K\in \nK_0^s(E)$ such that one of the assertions (2)--(6) holds, one can take the same set also in all other assertions.
\end{theorem}

\begin{proof}
(\ref{it:mulc}) $\Leftrightarrow$ (\ref{it:nulc}) holds by definition, and (\ref{it:nurestB}) $\Leftrightarrow$ (\ref{it:nurest}) follows by taking a suitable compact set $K$ by Lemma~\ref{lem:levymeasurecomparable}.

(\ref{it:fourier}) $\Rightarrow$ (\ref{it:nurestB}): In a Banach space, $\nu\|_{E_K}$ is a L\'evy measure if and only if $\f$ is the characteristic functional of a some probability measure, cf.\ \cite[Theorem~3.4.9]{heyer}, where $\f$ is evaluated in all $a\in {E_K}'$. The set $E'|_{E_K}$ separates the points of $E_K$ and therefore, $\f$ uniquely determines a measure on $\nB(E_K)=\nE(E_K,E'|_{E_K})$. This implies that $\nu\|_{E_K}$ and hence $(\nu|_K)\|_{E_K}$ is a L\'evy measure on $E_K$, assertion \eqref{it:nurestB}.

(\ref{it:nurestB}) $\Rightarrow$  (\ref{it:fourier}): Again using \cite[Theorem~3.4.9]{heyer}, it suffices to note $E'|_{E_K}\subseteq {E_K}'$ and $\nu(K^c)<\infty$.

(\ref{it:nulc}) $\Rightarrow$ (\ref{it:fourier}): 
By assumption, $\wt{\ee}(\nu|_K)(E\setminus E_K)=0$ and $\wt{\ee}(\nu|_K)$ and $\wt{\ee}(\nu|_K)\big\|_{E_K}$ have the same Fourier transform if we use the set of continuous functionals $E'$ generating $\nB(E)$ and $\nB(E_K)$. One obtains
\begin{align*}
\wh{\wt{\ee}(\nu|_K)\big\|_{E_K}}(a) = \exp\left(\int_{E_K}\ee^{\ii \<x,a\>}-1-\ii \<x,a\>\cf_{K}\dd \nu|_K(x)\right),\quad a\in E'|_{E_K}
\end{align*}
which yields \eqref{it:fourier} if one takes the convolution of this measure with the Poisson measure of $(\nu|_{E_K\setminus K})\|_{E_K}\in\nM^b(E_K)$.

(\ref{it:nurestB}) $\Rightarrow$ (\ref{it:lmdef}): If $\nu':=(\nu|_K)\|_{E_K}$ is a L\'evy measure on $E_K$, there exist $x_\r\in E_K\subseteq E$ such that the familiy $(\ee(\r)*\d_{x_\r})_{\r\leq \nu'}$ is uniformly tight, i.e., for every $\e$ there exists a $\|\cdot\|_K$-norm compact set $H^\e$ with $(\ee(\r)*\d_{x_\r})(H^\e)>1-\e$ for all $\r\leq \nu'$. As every compact set is bounded, there exists an $n(\e)\in\NN$ with $H^\e\subseteq n(\e)\cdot K$ which implies assertion (\ref{it:lmdef}).

(\ref{it:lmdef}) $\Rightarrow$ (\ref{it:nulc}): If $\nu$ is finite, $\ee(\nu)(E\setminus E_K)=0$ because $\nu^{*n}(E\setminus n\cdot K)=0$, due to $n\cdot K\subseteq E_K$.
So let $\nu$ be a L\'evy measure which is not finite. The measure $\nu|_K$ is a L\'evy measure on $E$. By assumption, there exist $x_\r\in E$ such that the familiy of shifted Poisson measures satisfies the $n\cdot K$-tightness condition in (6). By Lemma~\ref{lem:Lshift}, the shifts $x_\r$ can be taken in $E_K$ without loss of generality. As $\wt{\ee}(\nu|_K)$ is an accumulation point of the family $(\ee(\r)*\d_{x_\r})_{\r\leq \nu|_K}$, for all $\e>0$, all $f\in \nC_b(E)$ and all $\r_0$ there exists a measure $\r\geq \r_0$ such that
\begin{align}
\left|\wt{\ee}(\nu|_K)(f)-\ee(\r)*\d_{x_\r}(f)\right|<\e.
\end{align}
Assuming that there exists a Borel set $B\subseteq E\setminus E_K$ with positive measure $\wt{\ee}(\nu|_K)$, we find also compact set $C\subseteq B$ with positive measure. There exists a continuous function $g\colon E\rightarrow [0,1]$ such that $g= 0$ on the closed set $n\cdot K$ and $g=1$ on the compact set $C$ due to complete regularity of $E$.
We have $\wt{\ee}(\nu|_K)(g)\geq \wt{\ee}(\nu|_K)(C)>0$. Furthermore, we see that $\ee(\r)*\d_{x_\r}(E_K\setminus n\cdot K)\leq \e$ for $n$ large enough which can be chosen independently of $\r$ and $x_\r$ by assumption.
But for all $\e>0$ one finds a $\r$ and $n\in \NN$ such that
\begin{align}
\nonumber\left|\wt{\ee}(\nu|_K)(g)\right| &= \left| \wt{\ee}(\nu|_K)(g) - \ee(\r)*\d_{x_\r}(g) + \ee(\r)*\d_{x_\r}(g) \right|\\
&\leq \left|\wt{\ee}(\nu|_K)(g)-\ee(\r)*\d_{x_\r} (g)\right|+\left|\ee(\r)*\d_{x_\r} (g)\right|\leq \e+\e
\end{align}
for $g$ constructed as above. This is a contradiction to the claim that the generalised Poisson measure of $E\setminus E_K$ was positive and the proof is complete.
\end{proof}

The null extension of $\mu\in\nM(E_K)$ to $E$ is defined by $\mu^0 (B) :=\mu(B\cap E_K)$ for $B\in \nB(E)$. By the previous lemmas we can identify Gaussian, Poissonian and infinitely divisible measures on $E_K$ and $E$ by restriction resp.\ extension provided that $E\setminus E_K$ has measure zero and $K\in\nB_0^s(E)$. Furthermore, we can identify L\'evy measures on $E_K$ and $E$ if there is a $\nu$-reducing set $K\in\nK_0^s$ by Theorems~\ref{thm:restLevy}. 

A locally reducible L\'evy measure has a simpler characterisation than in Corollary~\ref{cor:levymeas}:
\begin{proposition}\label{prop:levyEKequivalences}
$K\in\nK_0^s(E)$ is $\nu$-reducing for a L\'evy measure $\nu\in\nM(E)$ if and only if the following is satisfied:
\begin{enumerate}[label=(\roman*)]
\item $\nu(\{0\})=0$,
\item $\nu\big|_{(\a K)^c}\in\nM^b(E)$ for some (all) $\a>0$, and
\item the family $\big\{\ee\big(\nu\big|_{K\setminus \d_n K}\big)\big\|_{E_K}:n\in\NN\big\}$ is shift tight (w.r.t.\ $E_K$) for every (some) positive null sequence $(\d_n)_n$.
\end{enumerate}
\end{proposition}
\begin{proof}
$\Longrightarrow$: If $K$ is $\nu$-reducing, $(\nu|_K)\|_{E_K}$ is a L\'evy measure on the separable Banach space $E_K$ by Theorem~\ref{thm:restLevy}. By \cite[Proposition~3.4.9]{heyer} and Prokhorov's theorem, we have that $(\nu|_K)\|_{E_K}$ satisfies (i) $(\nu|_K)\|_{E_K}(\{0\})=0$, (ii) $(\nu|_K)\|_{E_K}((\d K)^c)<\infty$ and (iii) for all (some) null sequences $\d:=(\d_n)_n$ the set $M(\d,K)$ is shift tight w.r.t.\ $E_K$. The assertion follows from taking null extensions to the original space and $\nu(K^c)<\infty$.\\
$\Longleftarrow$: Let $\nu$ satisfy (i)-(iii), then $(\nu|_K)\|_{E_K}$ is a L\'evy measure on $E_K$ by \cite[Proposition~3.4.9]{heyer}. In particular, $M(\d,K):=\{\ee(\nu|_{K \setminus \d_nK})\|_{E_K}:n\in\NN\}$ is shift tight in $E_K$ and, a fortiori w.r.t.\ $(E_K,\t)$ thus $E$. We note that $M':=M(\d,K)^0* \ee(\nu|_{K^c})$ (elementwise convolution) is shift tight in $E$: If for $\e\in (0,1)$ the set $K_{\e/2}\in\nK_0(E_K)\subseteq \nK_0(E)$ satisfies $\mu^0(E\setminus K_{\e/2})=\mu(E_K\setminus K_{\e/2})<\e/2$ for all $\mu\in M$, and $H_{\e/2}\in\nK_0(E)$ satisfies $\ee(\nu|_{K^c})(H_{\e/2}^c)<\e/2$, then
$$\mu* \ee(\nu|_{K^c})\big((K_{\e/2}+H_{\e/2})^c\big)<1-\mu(K_{\e/2})\ee(\nu|_{K^c})(H_{\e/2})<\e.$$
Furthermore, $\nu=\sup_{\mu \in M'} \mu\in\nM(E)$ which proves that $\nu$ is a L\'evy measure on $E$. The first remark gives local reducibility  of $\nu$ to $E_K$ with reducing set $K$.
\end{proof}

\subsection{Zero-One Laws and Reducibility}
We show that in spaces with a fundamental system of separable Banach disks, every L\'evy measure is locally reducible. For convenience, we quote a key result of Janssen.
\begin{theorem}[Janssen, {\cite[Theorem~9]{janssen_01}}]\label{thm:janssen}
Let $\nu\in\nM(E)$ be a L\'evy measure and $\mu:=\wt{\ee}(\nu)$ its generalised Poisson exponential. If $H$ is a measurable linear subspace of $E$ with $\nu(H^c)=0$ and $x\in E$, then $\mu(x+H)\in\{0,1\}$.
\end{theorem}
The following simple fact can be straightforwardly checked.
\begin{lemma}\label{lem:sumsep}
If $K_1,K_2\in \nK_0^s(E)$ then $K_1+K_2\in \nK_0^s(E)$.
\end{lemma}

\begin{proposition}\label{prop:Banachsupportzeroone}
If $E$ admits a fundamental system of compact separable Banach disks one has:
\begin{enumerate}
\item For every generalised Poisson measure $\mu$ with L\'evy measure $\nu$ satisfying $\nu(K_0^c)=0$, $K_0\in\nK_0(E)$, there exists $K\in\nK_0^s(E)$ such that $\mu(E_K)=1$.
\item Every infinitely divisible measure $\mu\in\nI(E)$ is locally reducible.
\end{enumerate}
\end{proposition}
\begin{proof}
(1) Let $\mu=\wt{\ee}(\nu)$. As $E$ has a fundamental system of compact separable Banach disks and $\mu$ is tight, we have that there exists $K_1\in \nK_0^s(E)$ with $\mu(E_{K_1})\geq \mu(K_1)>0$. Furthermore, there exists a set $K_2\in \nK_0^s(E)$ with $K_0\subseteq K_2$ such that $\nu(K_2^c)=0$. We set $K:=K_1+K_2\in\nK_0^s(E)$ by Lemma~\ref{lem:sumsep}. Then, we have that $\nu(E_K^c)\leq \nu(K_2^c)=0$ and $\mu(E_K)\geq \mu(K_1)>0$. Theorem~\ref{thm:janssen} implies that $\mu(E_K)=1$. In other words, $\mu$ is locally reducible.\\
(2) Let $\mu$ be infinitely divisible with L\'evy measure $\nu$. According to (1), there exists $K_0\in\nK_0(E)$ such that $\nu|_{K_0^c}$ is a finite measure, and a set $K\supseteq K_0$ with $K\in \nK_0^s(E)$ such that $\wt{\ee}(\nu|_{K_0})(E_K)=1$. Noting that
$$\wt{\ee}(\nu|_K)(E_K)= \ee(\nu|_{K\setminus K_0})*\wt{\ee}(\nu|_{K_0})(E_K+E_K)\geq \ee(\nu|_{K\setminus K_0})(E_K)\wt{\ee}(\nu|_{K_0})(E_K)=1$$
we obtain the assertion.
\end{proof}
\begin{theorem}\label{thm:Banachsupportzeroone}
If $E$ admits a fundamental system of separable Banach disks, every infinitely divisible measure $\mu\in\nI(E)$ is locally reducible.
\end{theorem}
Let $\nu$ be the L\'evy measure of $\mu$ and $K\in \nK_0(E)$ with $\nu(K^c)<\infty$. Without loss of generality assume that $\mu=\wt{\ee}(\nu)$. As $K_1+K_2\in\nK_0(E)$ for $K_1,K_2\in\nK_0$ by continuity of addition, we can assume that $\wt{\ee}(\nu|_K)(E_K)>0$ as in the proof of Proposition~\ref{prop:Banachsupportzeroone},~(2). Let now $B\in\nB_0^s(E)$ with $B\supseteq K$ which exists by assumption. We obtain that $\mu_B:=\wt{\ee}(\nu|_B)(E_B)>0$ and $\nu(B^c)<\infty$. In particular, $\mu_B(E_B)=1$ and $\mu_B\|_{E_B}\in\nI(E_B)$ by Theorem~\ref{thm:janssen} and Lemma~\ref{lem:infinitedivsubspace} with L\'evy measure $\nu_B=\nu\|_B$ by uniqueness of the L\'evy measure. As in the separable Banach space $E_B$ there is a a fundamental family of $\|\cdot\|_B$-compact separable Banach disks $\nK_0^s(E_B)$, we obtain that $\mu_B$ is locally reducible on $E_B$, i.e.\ there exists $H\in\nK_0^s(E_B)$ with $\wt{\ee}((\nu_B)|_H)(E_H)=1$ and $\nu_B(E_B\setminus H)<\infty$. The assertion follows by $H\in\nK_0^s(E_B)\subseteq \nK_0^s(E)$ and noting that $\mu=\mu_B*\ee(\nu|_{B^c})$.

The theorem above states in particular that this property depends only on the duality $\<E,E'\>$.

\begin{example}\label{ex:sepext}
The following complete locally convex Suslin spaces have a fundamental system of compact separable Banach disks: All separable Fr\'echet and Banach spaces, furthermore, the well-known spaces of test functions and distributions $\nD,\nD',\nE,\nE',\nO_M,\nO'_M,\nO_C,\nO'_C,\nS,\nS'$, cf.\ \cite[pp.\ 115-177 and 233]{schwartz_radon} and Appendix~\ref{app:sfa}. A separable Banach or Fr\'echet space with the weak topology satisfies the condition of Theorem~\ref{thm:Banachsupportzeroone} as the property of sets being bounded only depends on duality. 
Furthermore, finite direct sums and closed subspaces of spaces with a fundamental system of (compact) separable Banach disks share the same property, cf.\ Appendix~\ref{app:sfa}.

\end{example}

\paragraph*{Open questions.}~ If $E$ has a fundamental system of separable Banach disks, then all $\mu\in\nI(E)$ are locally reducible. However, the converse implication is not known.\medskip\\
We assume for the rest of the paper that $\PP_{X_1}=\mu_1$ is locally reducible.

\section{C\'adl\'ag functions in Suslin spaces}\label{sec:cadlagfunctionsinsuslinspaces}
Before we begin our investigations on random measures we present some results for \cadlag functions with values in locally convex Suslin spaces. The results are similar to those in \cite[Chapter~3]{billingsley}. But due to the possibly uncountable neighbourhood bases, they are not standard.

We denote by $\nD(T; E)$ the space of \cadlag functions $\xi\colon T\arre E$. This space will always be endowed with the $\s$-algebra $\nF_{\nD}$ of cylinder sets on $\nD(T;E)$ which are generated by the \emph{coordinate functions} $x_t\colon \nD(T;E)\rightarrow E$ defined by $x_t(\xi):=\xi(t)$, $t\in T$.  For $t\in T$ the \emph{left limit mapping} $x_{t-}\colon \nD(T;E)\rightarrow E$ is defined by $x_{t-}(\xi):=\lim_{s\arrne t}x_s(\xi)$ and the \emph{jump function} by $\D \xi_t:=x_t(\xi)-x_{t-}(\xi)$, $t\in T$.\smallskip

Given a dense subset $T_0\subseteq T$, \cadlag functions are defined as follows: $\xi\colon T_0\arre E$ is \cadlag if and only if for all increasing or decreasing Cauchy sequences $(t_n)_{n\in\NN}$ in $T_0$ the limits of $\xi(t_n)$ exist and if $t_n\arrse t\in T_0$ the limit equals $\xi(t)$. As above, $\nD(T_0;E)$ denotes the set of such \cadlag functions. $\xi\in \nD(T_0;E)$ is said to have a jump in $s\in T$, if the limits
$$ y_s:=\lim_{\substack{r\arrse s\\ r\in T_0}} \xi(r)\quad\text{and }\quad  y_{s-}:=\lim_{\substack{r\arrne s\\ r\in T_0}} \xi(r) $$
are different. $\D \xi(s):=y_s-y_{s-}$ is the jump size in $s$ and $s\mapsto \D \xi(s)$ is the jump function corresponding to $\xi$.

\begin{lemma}\label{lem:uniformity}
Let $T_0\subseteq [0,t_{\max}]$ be dense. For a \cadlag function $\xi\in \nD(T_0;E)$, a continuous seminorm $p$ and $\e>0$ there exist finitely many points $0=t_0<t_1<\ldots<t_n=t_{\max}$ such that
\begin{align}\label{eq:uniformity}
\sup\{p(\xi(t)-\xi(s))\colon s,t\in [t_{i-1},t_i)\cap T_0\}<\e,\quad i=1,\ldots,n.
\end{align}
Furthermore, for any weaker metric $d$ on $E$, there are finitely many points with
\begin{align*}
\sup\{d(\xi(t),\xi(s))\colon s,t\in [t_{i-1},t_i)\cap T_0\}<\e,\quad i=1,\ldots,n.
\end{align*}
\end{lemma}
\begin{proof}
Let $\pi\colon E\rightarrow E/p^{-1}(\{0\}))$ be the canonical projection associated to $p$ which is continuous. Furthermore, $\pi\circ \xi\in \nD(T;E/p^{-1}(\{0\})$ and all the limits exist by completeness of $E$. The target space is normed, but not necessarily complete. The expression in \eqref{eq:uniformity} remains the same if the \cadlag function is projected onto the quotient space. Exactly as in \cite[Chapter~3, Lemma~1, p.~110]{billingsley}, where completeness is not needed, one obtains the first assertion.\\
If $d$ is a continuous metric on $E$, $\id: (E,\t)\rightarrow (E,d)$ is continuous and by the same reasoning as above the assertion follows.
\end{proof}
\begin{lemma}\label{lem:countablejumps}
A \cadlag function with values in a complete locally convex Suslin space has at most countably many jumps.
\end{lemma}
\begin{proof}
Let $(E,\t)$ be the Suslin space with the original topology and $(E,d)$ a metric space, where $d$ is a weaker metric. Let $\id\colon (E,\t)\arre (E,d)$ be the continuous identity map. If $\xi\in \nD(T_0;E)$ has a jump in $t_0\in T$ then $\id \circ \xi $ has a jump in $t_0$. Thus, $\id\circ \xi$ has at least as many jumps as $\xi$ and is \cadlag as well by continuity of the identity map. Furthermore, if $\xi$ is continuous in $t_0$, this carries over to $\id\circ\xi$, thus the jumps of $\xi$ and $\id\circ \xi$ are the same. But for \cadlag functions with values in metric spaces it is well-known that there are at most countably many jumps \cite[Lemma~4.5.1]{ethierkurtz86}.
\end{proof}
Given $\xi\in\nD(T;E)$ one can number the jumps of $\xi$ in the following way:  Choose the metric $d$ of the proof above and measure the jumps as above by $d(\xi_t,\xi_{t-})$. Then, by Lemma~\ref{lem:uniformity}, there are only finitely many jumps on bounded intervals if $d(\xi_t,\xi_{t-})$ is larger than 1. Therefore, one can denote these jump times by $t_{1,1}(\xi),t_{1,2}(\xi),\ldots$ (setting $t_{1,k+1}(\xi)=t_{1,k+2}(\xi)=\ldots :=\infty$ if there are $k$ jumps larger than 1) and obtain thus a numbering.  If $d(\xi_t,\xi_{t-})\in (\frac{1}{n+1},\frac{1}{n}]$ for some $n\geq 1$, we can do the same procedure using $t_{n,1}(\xi),t_{n,2}(\xi),\ldots$ and so on.
\begin{lemma}\label{lem:jumpsmeasurable}
Given $n$ and $j$, the map $t_{n,j}\colon \nD(T;E)\arre T\cup \{\infty\}$ is $\nF_{\nD}$-measurable.
\end{lemma}
The proof is exactly the same as in \cite[Proof of Lemma~20.9]{sato:99} replacing the modulus by $d$ which is a continuous metric on $E$.

\begin{lemma}\label{lem:leftlim}
For $t\in T$, $t>0$, the left limit mapping $x_{t-}\colon \nD(T;E)\rightarrow E$ is $\nF_{\nD}-\nB(E)$-measurable.
\end{lemma}
\begin{proof}
Let $d$ be a weaker metric on $E$. We note that $\nB(E,\t)=\nB(E,d)$ and the Borel sets are generated by open $d$-balls $B_x(\e):=\{y\in E: d(x,y)<\e\}$, $x\in E$ due to $(E,d)$ being a separable metric space and therefore second countable. Furthermore, the $\s$-algebra $\nF_{\nD}$ on $\nD(T; (E,\t))$ only depends on $\nB(E)$ and not the precise topology generating it. We define $x_{t-}^d(\xi):=d-\lim_{s\arrne t} x_s(\xi)$ for $\xi \in\nD(T;(E,\t))$ and show that $x_{t-}^ d(\xi)=x_{t-}(\xi)$. But this follows from $(E,d)$ being Hausdorff and the continuity of the injection $(E,\t)\hookrightarrow (E,d)$ and therefore
$$x_{t-}(\xi)=\lim_{s\arrne t} x_s(\xi) = d-\lim_{s\arrne t} x_s(\xi)=x_{t-}^d(\xi)\quad \text{for }\xi\in \nD(T; (E,\t)).$$
We test measurability on a generator of $\nB(E)$ and obtain
\begin{align*}
&(x_{t-})^{-1}(B_x(\e)) = (x^d_{t-})^{-1}(B_x(\e)) \\
&= \big\{\xi \in \nD(T; (E,\t)):\, \exists n\in\NN\;\forall q\in\QQ\cap (0,1/n)\colon \; x_{t-q}(\xi)\in B_x(\e)\big\}\in \nF_{\nD},
\end{align*}
which proves the lemma.
\end{proof}

\begin{lemma}\label{lem:jointlymeasurable}
The maps $x\colon \nD(T;E)\times T\rightarrow E$ defined by $(\xi,t)\mapsto x_t(\xi)$ and $\conj{x}\colon \nD(T;E)\times T\rightarrow E$ defined by $(\xi,t)\mapsto x_{t-}(\xi)$ are $\nF_{\nD}\ten \nB(T)- \nB(E)$-measurable.
\end{lemma}
\begin{proof}
For all $t\in T$ the map $x_t\colon \nD(T;E)\rightarrow E$ is $\nF_{\nD}-\nB(E)$-measurable. Using right-continuity of $\xi$ in the $d$-topology we obtain that the $\nF_{\nD}\ten\nB(T)$-measurable maps
$$x_s^{(n)}(\xi):=\sum_{k=0}^{2^{n}-1}x_{t_{\max} (k+1)2^{-n}}(\xi) \cf_{\big(\frac{t_{\max}k}{2^n},\frac{t_{\max}(k+1)}{2^n}\big]}(s),\quad  \text{if }s\in T=[0,t_{\max}]$$
or
$$x_s^{(n)}(\xi):=\sum_{k=0}^{4^{n}-1}x_{(k+1)2^{-n}}(\xi) \cf_{\big(\frac{k}{2^n},\frac{(k+1)}{2^n}\big]}(s),\quad \text{if }s\in T=[0,\infty),$$
respectively, converge to $x_s(\xi)$ in the $\t$- and the $d$-topology as $n$ tends to infinity. The latter convergence yields the desired measurability of the joint map.\\
For $\conj{x}$ the proof is exactly the same taking the left endpoints $x_{t_{\max}k2^{-n}}$ resp.\ $x_{k2^{-n}}$ in the approximating sums.
\end{proof}

\section{Wiener processes in locally convex spaces}\label{sec:brownianmotion}

We consider Gaussian random variables with values in a locally convex space $E$. Probably the most comprehensive monograph concerning Gaussian measures on locally convex topological vector spaces is \cite{bogachev}, the following concepts are taken mainly from chapters~2 and~3. See also e.g. \cite{billingsley,maasphd}.
A random variable $X\colon \Omega \rightarrow E$ is called Gaussian, if for all $a\in E'$ the real valued random variables $\<X,a\>=a(X)$ are Gaussian. Let $\r$ be a Gaussian measure. A mapping $q\colon E\rightarrow \RR_+$ is a \emph{$\r$-measurable seminorm} if there exists a $\nB(E)$-measurable linear subspace $E_0\subseteq E$ with $\r(E_0)=1$ and such that the restriction $q|_{E_0}$ is a seminorm on $E_0$. Obviously, $\|\cdot\|_K$ is a $\r$-measurable seminorm if $\r(E_K)=1$.

\begin{lemma}[Zero-one law for Gaussian measures, {\cite[Theorem 2.5.5]{bogachev}}]
Let $E_0\subseteq E$ be a $\nB(E)$-measurable affine subspace and $\r$ a Gaussian measure. Then, $\r(E_0)\in \{0,1\}$.
\end{lemma}

For later use, we formulate a proposition about Gaussian measures which can be reduced to a separable Banach subspace.
\begin{proposition}\label{prop:GaussianEK}
Let $\r$ be a centered Gaussian measure on a locally convex Suslin space $E$. If there exists a set $K\in \nK_0^s(E)$ of positive measure, one has
\begin{enumerate}
\item $\r(E_K)=1$,
\item $\r$ is a Gaussian measure on $E_K$ and
\item there exists $\a>0$ such that $\displaystyle \int_{E_K} \ee^{\a \|x\|_K}\dd\r(x)<\infty$.
\end{enumerate}
In particular, if $E$ has a fundamental system of compact separable Banach disks, every Gaussian measure on $E$ has a Banach support.
\end{proposition}
The assertions follow from the above cited zero-one law, Lemma~\ref{lem:gaussianmeasurecomparable} and by Fernique's theorem for measurable seminorms, \cite[Theorem~2.8.5]{bogachev}. The last assertion guarantees the existence of all moments of Gaussian random variables, whenever there exists a compact set of positive measure. But this follows from $E$ being a Suslin, thus a Radon space.

Let $w_t$ be a real-valued standard Brownian motion, $\s\geq 0$ and $\g\in \RR$. We consider every real-valued process of the form $y_t=\s w_t+\g t$ a one-dimensional Brownian motion. 
\begin{definition}
A continuous L\'evy process with values in $E$ is called \emph{Wiener process}.
\end{definition}
We provide the following theorem characterising Wiener processes in $E$.
\begin{theorem}\label{thm:BMinE}
For a stochastic process $(W_t)_t$ with values in $E$ the following assertions are equivalent:
\begin{enumerate}
\item $(W_t)_t$ is a Wiener process.
\item $(W_t)_t$ is a continuous process with independent increments such that $\<W_t,a\>_t$ is a (possibly degenerate) one-dimensional Brownian motion for all $a\in E'$.
\item $(W_t)_t$ is a \cadlag L\'evy process with Gaussian distributed increments.
\item $(W_t)_t$ is a \cadlag Gaussian process determined by the mean
$\g\in E$ (i.e.\ $\g$ satisfies $\EE \<W_t,a\>=\<\g,a\>$ for all $a\in E'$) and a symmetric, positive semidefinite operator $Q\in \nL(E',E)$, where $E'$ is equipped with the Mackey topology $\t_{\mu}(E',E)$, such that
\begin{align*}
\EE \<W_t-\g t,a\>\<W_s-\g s,b\> = \<Q a,b\> \min\{s,t\}
\end{align*}
for all $s,t\in T$ and $a,b\in E'$.
\end{enumerate}
\end{theorem}

\begin{proof}
(1) $\Longrightarrow$ (2): If $W$ is a continuous L\'evy process, the same holds for $\<W,a\>$ and it is well-known that a continuous L\'evy process in $\RR$ is a Brownian motion with drift of the form $\s_a w_t + \g_a t$ with $\s_a\geq 0$ and $\g_a \in \RR$ and $w=(w_t)_t$ is a real-valued standard Brownian motion. \sloppypar
(2) $\Longrightarrow$ (3): If $W$ is continuous, it is c\'adl\'ag. For the stationary increments property we note that the distribution of $W_t-W_s$ equals the distribution of $W_{t-s}$ on the $\pi$-system of cylindrical sets, thus on $\nE(E)=\nB(E)$. Furthermore, for all $a\in E'$ we have that $\PP_{\<W_t,a\>}$ is real-valued Gaussian which means that $W_t$ is Gaussian distributed for all $t$. Gaussian distributed increments follow from stationarity.\sloppypar
(3) $\Longrightarrow$ (4)  If $W_t$ obeys a Gaussian law, there exists an element $\g_t\in E$ and a continuous operator $Q_t\colon E'\arre E$ (where $E'$ is equipped with the Mackey topology) such that $\<\g_t,a\>=\EE\<W_t,a\>$  and $\<Q_t a,b\>=\EE \<W_t-\g t,a\>\<W_t-\g t,b\>$ for all $a,b\in E'$ by cf.\ \cite[Lemma~3.2.1]{bogachev}. \\
For simplicity, assume for the following that $\g=0$. The independent increments property yields
$$\EE\<W_t,a\>\<W_s,b\>= \EE\<W_t-W_s,a\>\<W_s,b\>+\EE\<W_s,a\>\<W_s,b\>=\<Q_s a,b\>$$
and one obtains the equality
$$\<Q_t a,b\>= n\<Q_{t/n} a,b\> $$
for every $n\in\NN$ and $t\in T$ by writing $W_t$ as a telescoping sum over an equidistant time net and using independent and stationary increments.
Finally, one obtains $Q_t=tQ_1=tQ$ by $W_t$ being \cadlag and using
$$\<Q_t a,a\>=\lim_{s\in\QQ,s\arrse t}\EE\<W_s,a\>\<W_s,a\> = \lim_{s\in\QQ,s\arrse t}s\<Qa, a\>=t\<Qa, a\>$$
and polarisation. Similarly one obtains $\g_t=t \g_1=t \g$.\sloppypar
(4) $\Longrightarrow$ (3): It suffices to check independent and stationary increments of the given process, but this follows immediately by the covariance structure of the occuring Gaussian vectors.\sloppypar
(3) $\Longrightarrow$ (1): In the sequel, let
$$S_\QQ:=\begin{cases}(S\cap \QQ) \cup \{ \max S \}& \max S \text{ exists},\\ S\cap \QQ & \text{else}\end{cases}$$ for some interval $S\subseteq [0,\infty)$. Let $\r_t:=\PP_{W_t}$ and $H\subseteq E$ be an absolutely convex compact set with $\r_1(H)>0$. Without loss of generality we assume that $\EE W_t=0$ for all $t$, i.e.\ all Gaussian distributions are centered. Then, the zero-one law for Gaussian measures, cf.\  \cite[Theorem~2.5.5.]{bogachev}, implies that $\r_1(E_H)=1$, where $E_H$ is the linear hull of $H$. Furthermore, for every $m\in\NN$, one has
$$\r_m(E_H)=\r_1^{* m}(E_H) = \r^{*m}(E_H+\ldots+E_H)\geq \r(E_H)^m\geq 1,$$
which implies together with Lemma~\ref{lem:unifH} below that $\r_t(E_H)=1$ for all $t\in T$.
The compact sets $n\cdot H\subseteq E$ are metrizable, \cite[Propositions~A.1.7 and~A.3.16]{bogachev} and we denote by $d_n$ a metric inducing the same topology on $n\cdot H$.\sloppypar
We take $T=[0,t_0]$, $t_0\geq 0$. If the claim holds on all bounded intervals, it is true on $\RR_+$. First, let us prove that $\PP(W_t\in E_H,t\in T)=1$. As $\PP_{W_t}(E_H)=\r_t(E_H)=1$ for every $t\in T$, one has $\PP(W_t\in E_H\colon t\in T_\QQ)=1$. For an $\om$ in this set of full measure we consider the trajectory $t\mapsto W_t(\om)$. We follow \cite[Proof of Proposition~7.2.3]{bogachev} and show that there exists an $n(\om)\in \NN$ such that $\{W_t(\om): t\in T_\QQ\}\subseteq n(\om)\cdot H$. To this end, we define the (possibly infinity-valued) process $\eta_t:=\|W_t\|_H$. Let $d$ be a metric on $E$ inducing a weaker topology. Defining the continuous function $\id\colon (E,\t)\rightarrow (E,d)$, the set $Q:=\id(H)$ is compact and absolutely convex in $(E,d)$.  The measurable seminorm $\|\cdot\|_Q$ on $(E,d)$ satisfies $\|\cdot\|_Q=\sup_{i\in\NN} a_i$ with $a_i\in (E,d)'\subseteq E'$, cf.\ \cite[Problem~A.3.27]{bogachev}. Furthermore, $\|x\|_H=\|\id(x)\|_Q$ for all $x\in E$. As in \cite[Proof of Proposition~7.2.3]{bogachev} one deduces from $a_i(W_t)$ being a real-valued martingale that $\eta$ is a submartingale.

Due to Doob's inequality and Fernique's theorem (cf.\ \cite[Theorem~2.8.5]{bogachev}) one has
\begin{align*}
\EE \sup_{t\in T_\QQ} \eta_t^2 \leq 2\EE \eta_{t_0}^2 =2 \int_E \|x\|^2_H\dd \r_{t_0}(x) <\infty
\end{align*}
which implies for $\omega\in\Omega_0$, a set of measure one, that $\sup_{t\in T_\QQ}\|W_t(\omega)\|_H<\infty$. But this yields the existence of an $n(\omega)\in\NN$ with $W_t(\om)\in n(\om)\cdot H$ for all $t\in T_\QQ$.
The limit
\begin{align*}
W_t(\om) = \lim\limits_{\substack{s\arrse t\\ s\in T_\QQ}}W_s(\om)
\end{align*}
is an element of $n(\om)\cdot H$ by completeness. It follows that $\{W_t(\om):t\in T\}\subseteq n(\om)\cdot H\subseteq E_H$ for $\om$ in a set of measure one. \sloppypar
The second part of the proof follows an idea of \cite[Proposition~A.1]{geiss_ylinen}. In \cite[Proposition~5]{feyel_delapradelle_95} a probability space $(\Om',\nF',\PP')$ and a process $(X_t)_{t\in T}$ are constructed such that the latter has the same finite-dimensional distributions as $W$ and continuous trajectories in $E$. As above, it is argued that they are almost surely in $E_H$, where $H\in\nK_0(E)$ with $\r_1(H)>0$. Let
\begin{align*}
\Om'_0&:=\left\{\om \in \Om'\colon \{X_t(\om): t\in [0,t_0]_\QQ\}\subseteq E_H\right\} \text{ and}\\
\Om'_n&:=\left\{\om \in \Om'\colon \{X_t(\om): t\in [0,t_0]_\QQ\}\subseteq n\cdot H\right\},\quad n=1,2,\ldots
\end{align*}
The latter sets form an ascending chain $\Om'_n\subseteq \Om'_{n+1}$ and their union equals $\Om'_0$. By hypothesis, $t\mapsto X_t(\om)$ is continuous for $\om\in\Om'_0$ and it is uniformly continuous as a mapping from $[0,t_0]$ to $n(\om)\cdot H$. In particular, its restriction to $[0,t_0]_\QQ$ is uniformly continuous as well. The set
\begin{align*}
\nO'_n := \bigcap_{k=1}^\infty \bigcup_{m=1}^\infty \bigcap_{\substack{s,t\in  [0,t_0]_\QQ\\|s-t|<\frac{1}{m}}} \left\{\om\in\Om'_n\colon d_n(X_t(\om),X_s(\om))\leq \frac{1}{k}\right\}\in\nF'
\end{align*}
equals $\Om'_n$ and thus $\PP'(\nO'_n)=\PP'(\Om'_n)$.
This property depends only on the distribution of the process. Setting
\begin{align*}
\Om_n&:=\big\{\om \in \Om\colon \{X_t(\om): t\in [0,t_0]_\QQ\}\subseteq n\cdot H\big\},\quad n=1,2,\ldots
\end{align*}
one obtains that
\begin{align*}
\nO_n := \bigcap_{k=1}^\infty \bigcup_{m=1}^\infty \bigcap_{\substack{s,t\in  [0,t_0]_\QQ\\|s-t|<\frac{1}{m}}} \left\{\om\in\Om_n\colon d_n(W_t(\om),W_s(\om))\leq \frac{1}{k}\right\}\in\nF
\end{align*}
has the same measure, $\PP(\nO_n)=\PP'(\nO'_n)$ and the sets are increasing. Let $\nO_0:=\bigcup_{n=1}^\infty \nO_n$  the set of uniformly continuous paths from $[0,t_0]_\QQ$ to $E_H$. Then,
$$\PP(\nO_0)=\lim_{n\arre\infty} \PP(\nO_n) = \lim_{n\arre\infty} \PP'(\nO'_n)=\PP'(\Om_0')=1,$$
therefore, $W(\om)$ is uniformly continuous on $[0,t_0]_\QQ$ for almost all $\om\in\Om$. One can define a unique continuous extension of the uniformly continuous function $W(\om)\colon [0,t_0]_\QQ\arre n(\om)\cdot H$ to the closure $[0,t_0]$ of its domain by setting
\begin{align*}
\wt{W}_t(\om) := \begin{cases} W_t(\om), & t\in\QQ\cap [0,t_0],\text{ and }\om\in\nO_0,\\
\lim\limits_{\substack{s\in \QQ\cap [0,t_0],\\s\arre t}}W_s(\om), & t\in \QQ^c\cap [0,t_0],\text{ and } \om\in\nO_0,\\
0, & \text{else.} \end{cases}
\end{align*}
But as
\begin{align*}
W_{t-}(\om) =\lim_{\substack{s\in [0,t_0]_\QQ\\ s\arrne t}}\!\!W_s(\om) = \lim_{\substack{s\in [0,t_0]_\QQ\\s\arre t}}\!\!W_s(\om) = \wt{W}_t(\om) =\lim_{\substack{s\in [0,t_0]_\QQ\\ s\arrse t}}\!\!W_s(\om) = W_t(\om)
\end{align*}
for $\om\in\nO_0$ and for every $t$ by the \cadlag property of $W$, one obtains that $W$ has already a.s.\ continuous sample paths.
\end{proof}

\begin{lemma}\label{lem:unifH}
Let $(\r_t)_{0\leq t \leq 1}$ be a semigroup of Gaussian measures on $E$. Then,
\begin{enumerate}
\item $\r_t(H)\geq \r_1(H)$ for every measurable absolutely convex set $H$ and $t\in [0,1]$.
\item If $H$ is bounded and $\r_1(H)>0$, $\r_t(E_H)=1$ for all $t$.
\end{enumerate}
\end{lemma}
\begin{proof}
By the semigroup property we have $\r_t*\r_{1-t}=\r_1$. Therefore,
\begin{align*}
\r_1(H) \, = \r_{t}*\r_{1-t}(H) &=  \int_E \int_E \cf_H(x+y) \dd \r_t(x)\dd \r_{1-t}(y) \\
& = \int_E \r_t(H-y) \dd \r_{1-t}(y)\\
& \leq \int_E \r_t(H)\dd \r_{1-t}(y) = \r_t(H),
\end{align*}
because for a Gaussian measure $\mu$, $a\in E$ and every absolutely convex set $A$ one has $\mu(A+a)\leq \mu(A)$ by \cite[Theorem~2.8.10]{bogachev}. The second assertion immediately follows by the zero-one law for linear subspaces, cf.\ \cite[Theorem~2.5.5.]{bogachev}.\nocite{bogachev_gauss_lin}
\end{proof}

\section{Jump processes and random measures}\label{sec:jumpprocesses}
Let $X$ be a L\'evy process in $E$ and $\nu\in\nM(E)$ the L\'evy measure of $\mu_1=\PP_{X_1}$ and locally reducible. Let $K\in\nK^s_0$ be a $\nu$-reducing set. In particular, the properties of Proposition~\ref{prop:levyEKequivalences} are valid. The main results of this section will be Theorem~\ref{thm:comppoisint} and Proposition~\ref{prop:orginialprocesses}.
\subsection{A constructed process}
We begin with some general facts about Poisson random measures, which can be found in \cite[Chapters 19 and 20]{sato:99}. Given a $\s$-finite measure space $(\Theta,\nB,\rho)$ one can define a Poisson random measure $\{N'(B),B\in\nB\}$ on some probability space $(\Om',\nF',\PP')$ with intensity measure $\rho$. For now, we choose this measure space to be the finite measure space $(T\times E,\nB(T\times E),\lambda\ten \nu|_{K^c})$. Then, for $\om\in\Om'_0\in \nF'$ with $\PP'(\Om'_0)=1$, the measure $N'(\cdot,\om)$ is supported on a finite number of points of mass 1 on all bounded intervals, and $N'(\{s\}\times E,\omega)$ has values in $\{0,1\}$ for all $s\in T$, cf.\ \cite[Lemma~20.1]{sato:99}.

\begin{definition}
Let $B\in \nB(E)$. The Poisson integral with respect to $N'$ is defined by
\begin{align}\label{eq:PoissonInt}
Y_t(B)(\om) := \int_{(0,t]\times B} x \dd N'(s,x)(\om),\quad \om\in \Om'_0
\end{align}
and $Y_t(\om):=0$ for $\om\in \Om'^c_0$. If $B=E$, we write $Y_t:=Y_t(E)$.
\end{definition}
\begin{proposition}
\begin{enumerate}
\item For $t\in T$, $B\in \nB(E)$, the random variable $Y_t(B)$ defined in \eqref{eq:PoissonInt} is finite.
\item  $(Y_t(B))_t$ is a \cadlag L\'evy process in $E$ and $\PP'_{Y_t(B)}$ has distribution $\ee(t\nu|_{K^c\cap B})$. In particular, its characteristics are  $(0,0,t\nu|_{K^c\cap B},K)$.
\item $N'(B)(\om) = \# \{s: (s,\D Y_s(\om))\in B\setminus \{0\}\}$
\end{enumerate}
\end{proposition}
\begin{proof}
(1) follows from the property that $N'(\cdot,\om)$ is a.s.\ supported on a finite number of points. \\
(2) This essentially follows from \cite[Proposition~19.5]{sato:99}. For convenience, we adapt the situation to ours: The functional $a\in E'$ is measurable. We consider the real-valued random variable
\begin{align*}
Y_t^a(B)(\om):=\<Y_t(B)(\om),a\> = \int_{[0,t]\times B}\<x,a\>\dd N'(s,x)(\om),
\end{align*}
which has Fourier transform
\begin{align*}
\EE \ee^{\ii z Y^a_t(B)}=& \exp\left(t \int_{E} \big(\ee^{\ii z \<x,a\>}-1\big)\dd\nu|_{K^c\cap B}(x)\right) \\
=&\exp\left( \int_{[0,t]\times E} \big(\ee^{\ii z \<x,a\>\cf_B(x)}-1\big)\dd\lambda\ten \nu|_{K^c}(s,x)\right),\quad z\in \RR.
\end{align*}
Since this holds for arbitrary $a\in E'$, the $E$-valued random variable $Y_t(B)$ has distribution $\ee(t\nu|_{K^c\cap B})$ on $E$. Independent increments follow from $N'$ being independently scattered, stationary increments from the structure of the intensity measure. The map $t\mapsto \ee(t\nu|_{K^c\cap B})=\PP_{Y_t(B)}$ is a convolution semigroup. Weak continuity of the semigroup of distributions of $Y_t$ follows from
\begin{align*}
\left|\big [\ee(t\nu|_{K^c})-\d_0\big](f) \right| & \leq \left| \ee^{-t \nu|_{K^c}(E)} \sum_{n=1}^\infty \frac{t^n \nu|_{K^c}(E)^n M^n}{n!}\right|\\
&= \ee^{-t\nu|_{K^c}(E)}\left[ \ee^{t\nu|_{K^c}(E)M} -1 \right] \rightarrow 0\quad \text{for }t\arrse 0,
\end{align*}
where $f\in \nC^b(E)$ and $\mathrm{im}(f)\subseteq [-M,M]$. Weak continuity of $\PP_{Y_t(B)}$ is shown analogously substituting $K^c$ by $K^c\cap B$ in the previous calculation.
 \\
(3) If $s\in T$ such that $N'(\{s\}\times E,\om)=0$, then $\D Y_s(\om)=0$. If there is an $x\in B\setminus \{0\}$ such that $N'(\{(s,x)\},\om)=1$, then $\D Y_s(\om) = x$.
\end{proof}

For $A\in \nB(T\times E)$ satisfying $A\subseteq T\times (K\setminus \e K)$ for some $\e\in (0,1)$ define the Bochner integral
\begin{align}\label{eq:bochint}
\int_A x \dd (\lambda\ten \nu)(t,x).
\end{align}
Indeed, $(\lambda\ten\nu)|_A$ is finite and concentrated on (a subset of) $T\times K$. The map $\psi\colon T\times K\rightarrow E_{K}$, $(t,x)\mapsto x$, satisfies the conditions of Bochner integrability: It takes values in a Banach space $E_K\subseteq E$ and is Bochner measurable (which is equivalent to being measurable in separable Banach spaces); for Bochner integrals in locally convex spaces cf.\ \cite[Definition~5, p.~75]{thomasintegrationlcs}. Therefore, \eqref{eq:bochint} can be considered as a Bochner integral in $E_K$ or in $E$.

The measure $\lambda\ten \nu$ is $\sigma$-finite on $\nB(T\times E)$ using the partition $E=\bigcup_n C_n$ with $C_0=K^c$ and $C_n:=\frac{1}{n}K\setminus \frac{1}{n+1}K$ for $n=1,2,\ldots$. Indeed, $\nu(C_n)<\infty$ by Proposition~\ref{prop:levyEKequivalences}~(ii). By $\sigma$-finiteness of $\lambda\ten \nu$, one can construct a Poisson random measure $N'$ on a probability space $(\Omega',\nF',\PP')$ with intensity measure $\lambda\ten \nu$, cf.\ \cite[Proposition~19.4]{sato:99}. Without loss of generality, $N'$ from above is just the restriction of this new Poisson random measure to $T\times K^c$.

Let $B\in \nB(E)$, $B\subseteq C_n$ and $t\in T$. Analogously as in \eqref{eq:PoissonInt} the L\'evy processes
$$Y_t(B)(\om):=\int_{[0,t]\times B}x\dd N'(s,x)(\om)$$
are finite for almost all $\om\in \Om'$, contained in $E_K$ for $n\neq 0$ and Poisson distributed on $E$ and $E_K$ by Lemma~\ref{lem:poissonmeasurecomparable}.

 \begin{definition}\label{def:comppoisson}
Let $B\in \nB(E)$ with $B\subseteq C_n$ and $t\in T$. The \emph{compensated Poisson integral} is defined by
\begin{align*}
J'([0,t]\times B):= \int_{[0,t]\times B} \hspace{-3mm}x\dd\wt{N'}(t,x) := \int_{[0,t]\times B}\hspace{-3mm} x\dd N'(t,x) - \int_{[0,t]\times B}\hspace{-3mm} x \dd (\lambda\ten \nu)(t,x).
\end{align*}
\end{definition}
\begin{definition}
Let $Z^n\colon T\times \Omega\rightarrow E$, $n=1,2,\ldots$ be \cadlag stochastic processes. They are said to \emph{converge almost surely uniformly on bounded intervals of $T$}, if there exists a set $\Omega_0$ of measure one, such that for all $\om\in\Om_0$, for all seminorms $p$ generating the topology $\t$ on $E$ and all bounded intervals $T_0\subseteq T$ one has that
$$\sup_{t\in T_0}p\left(Z^n_t(\om)-Z^m_t(\om)\right)<\e$$
for every $\e>0$ and $n,m\in \NN$ large enough.
\end{definition}

\begin{theorem}[Compensated Poisson integral]\label{thm:comppoisint}
With the notation from above, for $t\in T$, the quantity
\begin{align}\label{eq:comppoisint}
J'_t:=\int_{(0,t]\times K} x\dd \wt{N'}(s,x):=\sum_{n=1}^\infty  J'([0,t]\times C_n)
\end{align}
is a series of independent random variables in $E_K$ and converges almost surely in $E_K$ and $E$. The convergence is uniform in $t$ on bounded intervals of $T$. Finally, $(J'_t)_{t\in T}$ is a \cadlag L\'evy process in $E$ with characteristics $(0,0,\nu|_{K},K)$.
\end{theorem}
\begin{remark}
We understand the process $J'_t$ in the following sense: $J'_t(\om)$ takes the value of the right-hand side for all $t\in T$, if $\om\in \Om'_2$ such that the series converges uniformly on bounded intervals of $T$ in $\om$; and $J'_t:=0$ for $\om\in \Om'^c_2$.
\end{remark}
\begin{proof} Define $x_n:=-\int_{C_n}x\dd \nu(x)\in E_K\subseteq E$. The summands $J'([0,1]\times C_n)$ have distributions $\ee(\nu|_{C_n})*\d_{x_n}$ on $E$ or, by restriction $\ee(\nu|_{C_n})\|_{E_K}*\d_{x_n}\in\nM^1(E_K)$. Their sums converge weakly to $\wt{\ee}((\nu|_K)\|_{E_K})\in\nM^1(E_K)$, cf.\ proof of  \cite[Theorem~3.4.9]{heyer}. By the same arguments as in~\cite[Proof of Theorem~2.1]{dettweiler}, convergence in distribution of the partial sums of the independent random variables yields a.s.\ convergence in $E_K$ (cf.\ \cite[Theorem~3.1.6]{heyer}). Similarly, it converges for every $t\in T$. $(J'_t)_{t\in T}$ is a L\'evy process in $E_K$ with characteristics $(0,0,(\nu|_K)\|_{E_K},K)$. Indeed, the independent and stationary increments follow from the definition of $N'$. The semigroup of distributions equals $t\mapsto\wt{\ee}(t(\nu|_K)\|_{E_K})$ and is weakly continuous, cf.\ \cite[Theorem~2.3.9]{heyer}. 
\\
By Lemma~\ref{lem:EKprops} we have that all random elements on $E_K$ are also measurable with respect to $\nB(E_K)$. As $K$ is bounded and closed, the injection $\imath\colon E_{K}\hookrightarrow E$ is continuous and thus the series converges a.s.\ in $E$. In order to establish \emph{uniform} convergence on bounded intervals of $T$, we fix $t\in T$, $t>0$ and again use the injection $\imath$. Continuity implies $p(\imath(x))\leq c_p\|x\|_K$ for all $x\in E_K$ and all seminorms $p$ generating $\t$ and suitable constants $c_p>0$. Therefore,
\begin{align*}
\sup_{s\in [0,t]} p\left(\imath\left(J'_s - \sum_{n=1}^N J'([0,s]\times C_n)\right)\right) &\leq c_{p}\sup_{s\in [0,t]}\left\|J'_s- \sum_{n=1}^N J'([0,s]\times C_n)\right\|_K
\end{align*}
where the right hand side converges in probability by \cite[Proof of Thm.~2.1]{dettweiler} and due to the fact that $J'$ is a L\'evy process in $E_K$.

Applying \cite[Lemma~20.4]{sato:99} (which can be proven in exactly the same way for Banach spaces) yields a.s.\ uniform convergence on bounded intervals of the sequence of processes in $E_K$, a \cadlag limiting process $J'$ in $E_K$ as $\nD([0,t];E_K)$ is closed under uniform convergence. The inequality yields uniform convergence in $E$ and \cadlag functions in $E$ as well injecting the $\nD([0,t];E_K)$ functions into $\nD([0,t];E)$ by virtue of $\xi\mapsto \imath\circ \xi$. It remains to state that $J'$ is a L\'evy process in $E$. Continuity of $\imath$ implies weak continuity of the semigroup $t\mapsto \wt{\ee}(t\nu|_K)$, the null extension of the familiy of distributions on $E_K$ and the proof is complete.
\end{proof}

As mentioned above, the generalised Poisson exponential is unique up to a convolution. From
now on we will use a certain representative given by the construction above in order
to avoid ambiguities: We define $\wt{\ee}(\nu|_K)$ as the weak limit of $\wt{\ee}(\nu|_{C_n} ) * \delta_{x_n}$ with $x_n$ from the previous proof.

Given an infinitely divisible distribution $\r$ with characteristics $(\g,Q,0,K)$ on a locally convex Suslin space, where $Q\colon E'\arre E$ is a covariance operator associated to a Gaussian measure $\r$, one can define a Wiener process $W'_t$ on a probability space $(\Om',\nF',\PP')$ such that $W'_t\sim \r^{*t}$, cf.\ \cite[Proposition~5]{feyel_delapradelle_95}. The stochastic process $(W'_t)_{t}$ is continuous, has values in $E$ and
$$\EE \<W'_t-\g t,a\>\<W'_s-\g s,b\>=\min \{s,t\} \<Qa,b\>.$$
For details, we confer to Theorem~\ref{thm:BMinE}.
\begin{proposition}\label{prop:constructedprocess}
Let $X'_t=J'_t+L'_t+W'_t$, $t\in T$, be defined on a complete probability space $(\Om',\nF',\PP')$ where $L'_t:=Y_t(K^c)$ is defined as in \eqref{eq:PoissonInt}, and let the summands be constructed in a way such that they are independent. Then,
\begin{enumerate}
\item $(X'_t)_{t\in T}$ is a \cadlag L\'evy process and
\item $X'_1$ has characteristics $(\g,Q,\nu,K)$.
\item The random measure $N'$ satisfies $N'(B)=\#\{s: (s,\D X'_s)\in B\setminus \{0\}\}$ for $B\in\nB(T\times E)$.
\end{enumerate}
\end{proposition}
\begin{proof}
(1) is clear. (2) One obtains the characteristics by convolution of the ingredients and noting that $\wt{\ee}(\nu_1)*\ee(\nu_2)= \wt{\ee}(\nu_1+\nu_2)$ (as can be seen using the Fourier transform) for a L\'evy measure $\nu_1$ and $\nu_2\in\nM^b(E)$. In our situation, $\nu_1=\nu|_{K}$, $\nu_2=\nu|_{K^c}$.\\
(3) This property holds for each $\e>0$ if $B\cap (T\times \e K)=\emptyset$. Observing
$$B\setminus \{0\}=\bigcup_{n=0}^\infty (T\times C_n)\cap B$$
with $C_n$, $n\in \NN^*$, as defined above (and $C_0:=E\setminus K$), one has
$$ N'(B)=\sum_{n=0}^\infty N'((T\times C_n)\cap B),\quad \as$$
and the assertion follows.
\end{proof}

\subsection{Definitions on the original space}\label{subsec:definitions}
This section is following the ideas of Sato in \cite[Section~20]{sato:99} for the proof of the L\'evy-It\^o decomposition of finite dimensional L\'evy processes. Extra considerations  for measurability issues have been carried out in lemmas~\ref{lem:leftlim} and~\ref{lem:jointlymeasurable} above.
\begin{definition}\label{defin:randommeasures}
Let $B\in\nB(T\times E)$ and $X=(X_t)_{t\in T}$ the (original) L\'evy process given on $(\Om,\nF,\PP)$. Let $\xi\in \nD(T;E)$ be a \cadlag function and $\D x_t(\xi):= x_t(\xi) -x_{t-}(\xi)$. For $\omega\in\Om$ resp.\ $\xi\in\nD(T;E)$ define
\begin{align*}
N(B,\omega)&:=\#\big\{s: (s,\Delta X_s(\omega))\in B\setminus \{0\}\big\}\quad\text{resp.}\\
n(B,\xi)&:=\#\big\{s: (s,\D x_t(\xi))\in B\setminus \{0\}\big\}
\end{align*}
and for all sets $B$ with $\lambda\ten \nu(B)<\infty$ set
\begin{align}\label{eq:cprm}
\wt{N}(B,\omega)&:=N(B,\omega)-(\lambda\ten\nu)(B)\\
\wt{n}(B,\xi)&:=n(B,\xi)-(\lambda\ten\nu)(B).\nonumber
\end{align}
$N$ is called the \emph{Poisson random measure associated to $X$} and $\wt{N}$ the \emph{compensated Poisson random measure corresponding to $X$}.
\end{definition}

Let $(\Om,\nF,\PP)$ be the original probability space and $(\Om',\nF',\PP')$ the probability space of the constructed process of the previous section. Both processes $X$ and $X'$ consist of \cadlag paths only by construction. We define the following maps to the space of \cadlag functions:
\begin{align*}
\psi\colon & \Om\arre \nD(T;E), & \psi (\om):= X_\cdot(\om),\\
\psi'\colon & \Om'\arre \nD(T;E), & \psi' (\om) := X'_\cdot(\om).
\end{align*}
One obtains $\PP\circ \psi^{-1}=\PP'\circ (\psi')^{-1}$, where this measure, call it $\PP^\nD$, is defined on $\nF_\nD$.  This is due to equality of distributions of $X'_t$ and $X_t$ on $E$ and therefore of $X'$ and $X$ on $\nD(T;E)$. The following lemma shows that $\lambda\ten \nu$ is the compensator of $N$ and $n$.
\begin{lemma}\label{lem:PoisRM}
Let $B\in \nB(T\times E)$ with $\lambda\ten \nu(B)<\infty$. Then,
$$\PP_{N(B)} = \PP'_{N'(B)}=\PP^{\nD}_{n(B)}.$$
In particular, $N$ and $n$ are a Poisson random measures with intensity measure $\lambda\ten \nu$.
\end{lemma}
\begin{proof}
We follow the proof of Sato, cf.\ \cite[p.~132]{sato:99}. We have that $n(B,\psi(\om))=N(B,\om)$ and $n(B,\psi'(\om))=N'(B,\om)$ and furthermore they are equal in law, provided that they are $\nF_\nD$-measurable. But this follows from lemmas~\ref{lem:jumpsmeasurable} and \ref{lem:leftlim} and the fact that  $x_{t_{n,j}(\xi)}(\xi)$ and $x_{t_{n,j}(\xi)-}(\xi)$ are $\nF_\nD$-measurable due to Lemma~\ref{lem:jointlymeasurable}, thus
$$G(m,j):=\big\{\xi\in \nD(T;E): t_{m,j}<\infty \text{ and } \D x_{t_{m,j}(\xi)}(\xi)\in B\big\}\in\nF_\nD.$$
Noting that $n(B,\xi)$ can be written as the $m,j$-series over indicator functions of $G(m,j)$ yields the assertion.
\end{proof}

For a Borel set $A\in \nB(E)$ let $\Om_0$ be the intersection of all subsets $\Om_0^m$ of $\Om$ of full measure such that $N([0,m]\times A,\om)<\infty$, $m\in \NN$, if $T=[0,\infty)$. In the case $T=[0,t_{\max}]$, the set $\Om_0$ consists of all $\om\in \Om$ such that $N([0,t_{\max}]\times A,\om)<\infty$. For $\om\in \Om_0$ define
\begin{align}\label{eq:XtA}
X_t(A)(\om):=\sum_{0\leq s\leq t} \D X_s(\om) \cf_{A\setminus \{0\}}(\D X_s(\om))
\end{align}
and if $\om\in \Om_0^c$ we set the trajectory $X_\cdot(A)(\om)$ to zero. Furthermore, carrying out the same construction with $n$ instead of $N$, one can define
\begin{align*}
\quad x_t(A)(\xi):=\sum_{0\leq s\leq t}\D x_s(\xi)\cf_{A\setminus \{0\}}(\D x_s(\xi)),
\end{align*}
for $\xi\in \nD_0$ of full measure $\PP^\nD$ and also $X'_t(A)(\om)$ on a set $\Om'_0\in \nF'$ with $\PP'(\Om'_0)=1$ setting the whole trajectories to zero on the complements $\nD_0^c$ and $\Om'^c_0$. \medskip\\
Then, for all $\om\in \Om$ and $\om'\in \Om'$, it holds that
\begin{align*}
X_t(A)(\om)=x_t(A)(\psi(\om))\quad &\text{ and }\quad X'_t(A)(\om')=x_t(A)(\psi'(\om')).
\end{align*}
Letting
\begin{align}\label{eq:Ltdef}
L_t:=X_t(K^c)\quad \text{and}\quad l_t:=x_t(K^c)
\end{align}
we have
\begin{align}\label{eq:Ltcorrespondence}
L_t(\om)=l_t(\psi(\om))\quad &\text{ and }\quad L'_t(\om)=l_t(\psi'(\om))\quad \text{for all }\om\in\Om,\,\om'\in\Om'.
\end{align}

\begin{lemma}\label{lem:Ltrep}
Let $L_t$ be defined as in \eqref{eq:Ltdef} and the random measures $N$ and $n$ be defined as in Definition~\ref{defin:randommeasures}. Then, for almost all $\omega\in \Omega$ and $\xi\in\nD(T;E)$ one has
$$ L_t(\om)=\int_{(0,t]\times K^c}x\dd N(s,x)(\om)\quad\text{and}\quad l_t(\xi)=\int_{(0,t]\times K^c}x\dd n(s,x)(\xi)$$
and $L_t$ and $l_t$ are \cadlag L\'evy processes with characteristics $(0,0,\nu|_{K^c},K)$.
\end{lemma}
\begin{proof}
From $\lambda\ten \nu([0,t]\times K^c)<\infty $ we deduce $N([0,t]\times K^c,\om)<\infty$ for $\om\in\Om_1^t$ for some $\Om_1^t\in\nF$ with $\PP(\Om_1)=1$ by Lemma~\ref{lem:PoisRM}. For such $\omega\in \Om_1^t$, there are $s_1(\om),\ldots,s_{m(\om)}(\om)\in [0,t]$ with $N(\{s_k(\om)\}\times K^c,\om)=1$ and for all other $s\in [0,t]$ one has that $N(\{s\}\times K^c,\om)=0$. Therefore one obtains
$$\int_{[0,t]\times K^c} x\dd N (s, x)(\om) = \sum_{k=1}^{m(\om)} \D X_{s_k(\om)}(\om) =  \sum_{s\in [0,t]}\D X_{s}(\om) \cf_{K^c}(\D X_s(\om)).$$
If
$$\om\in\Om_1 := \bigcap_{m=1}^\infty \Om_1^m,$$
the above expression exists for every $t\in T=[0,\infty)$. If $T=[0,t_{\max}]$, $\Om_1:=\Om_1^{t_{\max}}$.
By \eqref{eq:Ltcorrespondence} the same follows for $l_t$ taking $\xi \in \nD_1$ with $\PP^{\nD}$-measure one.

The distributions of $L_t$, $l_t$ and $L_t'$ are the same by \eqref{eq:Ltcorrespondence}. So the weakly continuous semigroup of distributions of $L'$ carries over to $L$ and $l$. Therefore, the finite-dimensional distributions on the path space of $(L_t)_{t\in T}$, $(l_t)_{t\in T}$ and $(L'_t)_{t\in T}$ coincide which implies that the processes $L$ and $l$ have independent and stationary increments. Furthermore, $L_t$ is \cadlag by construction.
\end{proof}

Also the compensated integral with respect to $N$ resp.\ $n$ can be constructed analogously as above by setting
\begin{align*}
J([0,t]\times C_m)&:=\int_{(0,t]\times C_m} x \dd N(s,x) - \int_{(0,t]\times C_m} x\dd (\lambda\ten  \nu)(s,x)\quad \text{ resp.}\\
j([0,t]\times C_m)&:=\int_{(0,t]\times C_m} x \dd n(s,x) - \int_{(0,t]\times C_m} x\dd (\lambda\ten  \nu)(s,x)
\end{align*}
for $m\in\NN$. Again,
\begin{align*}
J([0,t]\times C_m)(\om) &= j([0,t]\times C_m)(\psi(\om))\quad \text{and}\\
J'([0,t]\times C_m)(\om')&=j([0,t]\times C_m)(\psi'(\om'))
\end{align*}
for all $\om\in\Om$ resp.\ $\om'\in\Om'$.
\begin{proposition}\label{prop:orginialprocesses}
Defining
\begin{align}\label{eq:Jtdef}
J_t&: = \int_{(0,t]\times K} x \dd \wt{N}(s,x):=\sum_{m=1}^\infty  J([0,t]\times C_m)\quad\text{ resp.}\\
j_t&:=  \int_{(0,t]\times K} x \dd \wt{n}(s,x):=\sum_{m=1}^\infty  j([0,t]\times C_m)\label{eq:jtdef}
\end{align}
one has the following:
\begin{enumerate}
\item The series \eqref{eq:Jtdef} resp.\ \eqref{eq:jtdef} converge almost surely in $E_K$ (thus in $E$) uniformly in $t$ on bounded intervals.
\item $J_t(\om)=j_t(\psi(\om))$ and $J'_t(\om')=j_t(\psi'(\om'))$ for almost all $\om\in\Om$ resp.\ $\om\in\Om'$.
\item $J_t$ and $j_t$ are \cadlag L\'evy process with characteristics $(0,0,\nu|_{K},K)$.
\end{enumerate}
\end{proposition}
\begin{proof} Analogously to Lemma~\ref{lem:Ltrep}, one obtains that the processes
$$J_t^m:=\sum_{k=1}^m J([0,t]\times C_k),\; j_t^m:=\sum_{k=1}^m j([0,t]\times C_k),\; \text{and}\; (J'_t)^m:=\sum_{k=1}^m J'([0,t]\times C_k)$$
are equal in distribution, where it has been proved above that the distributions of $(J'_t)^m$, $m=1,2,\ldots$ converge to an infinitely divisible measure with characteristics $(0,0,t\nu|_{K},K)$ for $m\rightarrow \infty$ and $t\in T$. \medskip \\
(1) Let $\Om'_2$ be the set of full $\PP'$-measure such that the series \eqref{eq:comppoisint} for $J'_t$ converges in Theorem~\ref{thm:comppoisint}. It equals the intersection over all sets $\Om'_2(N)$ of $\PP'$-measure one ($N=1,2,\ldots$) where
$$\Om'_2(N):=\Big\{\om\in\Om': \lim_{m\rightarrow \infty} \sup_{l,k\geq m}\sup_{t\in [0,N]}\big\| (J'_t)^l(\om)-(J'_t)^k(\om)\big\|_K=0\Big\}.$$
Let $\nD_2$ resp.\ $\Om_2$ be the intersection of sets $\nD_2(N)$ resp.\ $\Om_2(N)$ from above with $\Om'$ replaced by $\nD(T;E)$ resp.\ $\Om$ and $J'$ by $j$ and $J$, respectively. Then,
$$1=\PP'(\Om'_2) = \PP^{\nD}(\nD_2) = \PP(\Om_2)$$
as it only depends on the distribution of the processes. This yields a.s.\ uniform convergence on bounded intervals of $j_t$ and $J_t$ in $E_K$ and consequently in $E$. \smallskip\\
(2) For all $\om\in\Om'_2$ and all $n\in\NN$ we have $(J'_t)^n(\om)=j_t^n(\psi(\om))$. As the left-hand side and therefore the right-hand side converges uniformly we obtain that $\psi(\om)\in\nD_2$ and $j^n_t(\psi(\om))\rightarrow j_t(\psi(\om))$ uniformly in $t$ on bounded intervals by definition of $\nD_2$ and therefore $J'(\om)=j(\psi'(\om))$ for $\om\in\Om'_2$. The same arguments hold for $J$ on $\Om$. \smallskip\\
(3) $J$ is \cadlag by uniform convergence and $J$ and $j$ are L\'evy processes by equality of their finite dimensional distributions with those of $J'$.
\end{proof}
Finally, we set $Y_t(\om):=X_t(\om)-L_t(\om)-J_t(\om)$ and $y_t(\xi):=x_t(\xi)-l_t(\xi)-j_t(\xi)$ and $Y'_t(\om):=X'_t(\om)-L'_t(\om)-J'_t(\om) = W'_t(\om)+\g t$.
From the equality of finite-dimensional distributions of $(J,L,Y)$, $(j,l,y)$ and $(J',L',Y')$ we obtain:
\begin{proposition}\label{prop:independence}
The three processes $(J_t,L_t,Y_t)_{t\in T}$ are independent.
\end{proposition}

\section{L\'evy-It\^o-decomposition}\label{sec:levyito}
\begin{theorem}[L\'evy-It\^o-decomposition]\label{thm:levyito}
Let $X$ be an $E$-valued L\'evy process with characteristics $(\g,Q,\nu,K)$ and $\nu$ locally reducible with reducing set $K$. Then there exist an $E$-valued Wiener process $(W_t)_{t\in T}$ with covariance operator $Q$, an independently scattered Poisson random measure $N$ on $T \times E$ with compensator $\lambda\ten \nu$ and a set $\Om_0\in \nF$ with $\PP(\Om_0)=1$ such that for all $\omega\in\Om_0$ one has
\begin{align}
X_t(\om)=\g t \;+ \;W_t(\om) \;+\!\! \int\limits_{[0,t]\times K}\!\!\!x\dd \wt{N}(s,x)(\om)\;+\!\!\int\limits_{[0,t]\times K^c}\!\!\! x\dd N(s,x)(\om)\label{eq:levyitolcs}
\end{align}
for all $t\in T$. Furthermore, all the summands in \eqref{eq:levyitolcs} are independent and the convergence of the first integral in the sense of \eqref{eq:Jtdef} is a.s.\ uniform in $t$ on bounded intervals in $E_K$ and $E$.
\end{theorem}
\begin{proof} As the distribution $\mu_1$ of $X_1$ is locally reducible with reducing set $K\in\nK_0^s(E)$ it holds that $\nu(K^c)<\infty$ and $\nu|_K$ is a L\'evy measure on $E_K$.\sloppypar
Let $L_t$ and $J_t$ be defined as in \eqref{eq:Ltdef} and \eqref{eq:Jtdef}, respectively. By means of the mappings $\psi$ and $\psi'$ from Section~\ref{subsec:definitions}, the processes $(X_t,L_t,J_t)_{t\in T}$ and $(X'_t,L'_t,J'_t)_{t\in T}$ have the same finite-dimensional distributions. Therefore, the process $Z=(Z_t)_{t\in T}$ defined by $Z_t:=X_t-L_t$ has the same distribution as $X'_t-L'_t=J'_t+W'_t+\g t$ and is therefore a L\'evy process. The process $(Z_t)_{t\in T}$ has jumps of a size in $K$ and its characteristics are $(\g ,Q,\nu|_K,K)$, which is obtained by the characteristics of $X'_t-L'_t$.

A.s.\ uniform convergence of the series in \eqref{eq:Jtdef} and the fact that for every jump of $Z_t$ there exists an $n\in\NN$ such that the jump exceeds $n^{-1}\cdot K$ imply that for almost all $\om\in\Om$ the trajectory $Y_t(\om):=Z_t(\om)-J_t(\om)$ is continuous as all jumps are erased. The distribution of $J_t$ equals $\wt{\ee}(t\nu|_K)$ on $E$ by Theorem~\ref{thm:comppoisint}. Furthermore, $Y_t$ has characteristics $(\g,Q,0,K)$, namely the same as those of $X'_t-L'_t-J'_t$.

In addition, the continuous process $Y_t$ has stationary and independent increments, thus it is a continuous L\'evy process, a Wiener process with drift in $E$ by Theorem~\ref{thm:BMinE}.
The element $\g\in E$ satisfies $\EE\<Y_1,a\>=\<\g,a\>$ for all $a\in E'$. This is nothing else than $\g=\EE Y_1$ in the Pettis sense. Setting $W_t:=Y_t-\g t$ we obtain a centered Wiener process $W_t$. \sloppypar
Independence of the summands follows from Proposition~\ref{prop:independence}.
\end{proof}
We can even strengthen the result and let the process $X_t^0:=\g t + W_t+ J_t$ live in a Banach space $E_K$. In order to obtain this, we use Lemma~\ref{lem:sumsep}.

\begin{corollary}\label{cor:Xt0}
Let $E$ have a fundamental system of separable Banach disks. Then, there exists $K\in \nK_0^s$ such that $X_t=X_t^0+L_t$ and additionally,
\begin{enumerate}
\item $X_t^0$ has values in $E_K$ for all $t\in T$ a.s.\  and
\item $X_t^0$ is Bochner integrable and square integrable in $E_K$ for every $t\in T$.
\end{enumerate}
\end{corollary}
\begin{proof}
Define $K:=K_1+K_2+K_3$, where $K_1\in \nK_0^s$ is $\nu$-reducing which exists due to Theorem~\ref{thm:Banachsupportzeroone}, the set $K_2\in \nK_0^s$ has positive (centered) Gaussian measure and the whole trajectory of $W$ stays in $E_{K_2}$ by the same arguments as in the proof of Theorem~\ref{thm:BMinE}. Indeed, $K_2$ exists, as there must be $K'\in\nK_0$ of positive measure, therefore $B\in\nB_0^s$ with $B\supseteq K'$ as $\nB_0^s$ is fundamental. The separable Banch space $E_B$ has full Gaussian measure. As $\nK_0^s(E_B)$ is fundamental in $E_B$, there exists $K_2\in\nK_0^s(E_B)\subseteq \nK_0^s(E)$ of positive Gaussian measure, therefore $E_{K_2}$ has full Gaussian measure. 

The set $K_3$ is the absolutely convex hull of $\g\in E$ and $E_{K_3}\cong \RR$ is clearly separable. Lemma~\ref{lem:sumsep} yields that $K\in \nK_0^s$. Note that $K$ is $\nu$-reducing by Lemma~\ref{lem:immediate}, $K$ has positive centered Gaussian measure (and therefore measure one by \cite[Theorem~2.5.5]{bogachev}) and $\g\in K$. Therefore, $X_t^0$ has values in $E_K$ a.s.\ and can be actually considered as a process in the Banach space $E_K$ by analogous arguments as above. (Square) integrability follows from Proposition~\ref{prop:GaussianEK} and  \cite[Corollary~3.4]{perezabreu_rochaarteaga_2003}.
\end{proof}

\begin{proposition}\label{prop:eqBanach}
Let $E$ be a space with a fundemental system of separable Banach disks and $(X_t)_{t\in T}$ a L\'evy process with values in $E$ and characteristics $(\g,Q,\nu,K)$. The following assertions are equivalent:
\begin{enumerate}
\item There exists a $t_0\in T\setminus \{0\}$ such that $X_{t_0}$ takes values a.s.\ in a separable Banach space $E_1$ with closed unit ball compact in $E$.
\item There exists a $t_0\in T\setminus \{0\}$ such that $\PP_{X_{t_0}}$ has a Banach support $E_1$ with closed unit ball compact in $E$.
\item For all $t\in T$ one has $X_t\in E_1$ a.s.\ for a separable Banach space $E_1$ with closed unit ball compact in $E$.
\item For all $t\in T$ the distribution $\PP_{X_t}$ has a Banach support $E_1$ with closed unit ball compact in $E$.
\item $\nu$ has a Banach support $E_2$ with closed unit ball compact in $E$.
\end{enumerate}
Given (5), one can choose $E_1=E_K+E_2+E_3+\RR\g$, where $E_3$ is a suitable Banach support of the Gaussian part of $X$ and $K$ is $\nu$-reducing.
\end{proposition}
\begin{proof}
The equivalence of (1) and (2) resp.\ (3) and (4) and the implication (4) $\Rightarrow$ (1) are obvious and
(2) $\Rightarrow$ (4) follows from Lemma~\ref{lem:infinitedivsubspace}.\smallskip

(2) $\Rightarrow$ (5):  If (2) holds, $\PP_{X_{t_0}}$ is infinitely divisible on the Banach space $E_1$ and thus, there exists a L\'evy measure $\nu'$ on $E_1$. Injecting $E_1$ into $E$, finding the null extensions of $\PP_{X_{t_0}}$ and $\nu'$ and using Theorem~\ref{thm:restLevy} and uniqueness of the L\'evy measure of an infinitely divisible distribution, one obtains $\nu'^0=\nu$ and therefore, $\nu$ has Banach support $E_1$. One can choose $E_2=E_1$ and obtain assertion (5). \smallskip

(5) $\Rightarrow$ (2): Let $K\in\nK_0^s$ be $\nu$-reducing. Then, $\wt{\ee}(\nu)=\wt{\ee}(\nu|_K)*\ee(\nu|_{K^c})$ has Banach support $E_K+E_2$. Let $K_2\in \nK_0^s$ be the closed unit ball of $E_2$. The Gaussian part $\r$ has Banach support $E_3=E_H$ for some $H\in\nK_0^s$ and $\d_\g$ has Banach support $\RR \g$. Put $K':=K+K_2+H+[-1,1]\cdot \g$ which is in $\nK_0^s$ by Lemma~\ref{lem:sumsep}. Then, $\mu:=\wt{\ee}(\nu)*\r*\d_\g$ has Banach support $E_1:=E_{K'}=E_K+E_2+E_3+\RR \g$ and as $\mu\|_{E_1}$ is infinitely divisible on $E_1$, there exists a root $(\mu\|_{E_1})^{*t_0}=(\mu^{*t_0})\|_{E_1}$ by Lemma~\ref{lem:infinitedivsubspace}. Therefore, $\PP_{X_{t_0}}=\mu_{t_0}=\mu^{*t_0}$ also has Banach support $E_1$  which is the assertion.
\end{proof}

%
%
\appendix

\section{Some functional analysis}\label{app:sfa}
In this appendix we investigate conditions for $\nK_0^s(E)$ or $\nB_0^s(E)$ being fundamental in $\nK_0(E)$.
\begin{proposition}\label{prop:bacls}
In Fr\'echet spaces $\nK_0^s(E)$ is fundamental in $\nK_0(E)$.
\end{proposition}
\begin{proof}
The proof is essentially given in \cite[Theorem~3.6.5]{bogachev} but presented here for convenience: If $K\in \nK_0$ there exists $A\in \nK_0$ such that $K$ is compact in $E_A$, cf.\ \cite[Lemma~p.18]{nuclearconuclear}. Again, in $E_A$ there exists $C\in\nK_0(E_A)\subseteq \nK_0(E)$ with $K$ compact in $E_C$ and $C$ is compact in $E_A$. A factorisation lemma of Davis-Fiegel-Johnson-Pelchi{\'n}sky \cite[Corollary~1]{davis:figiel:74} for weakly compact operators provides a Banach space $Y$ which is reflexive and continuously embedded into $E_A$ and $C$ is bounded in $Y$. This continuity of $E_C\rightarrow Y$ yields that $K$ is compact in $Y$. Taking the closure $L$ in $Y$ of the linear hull of $K$ yields a reflexive subspace of $Y$ which is separable as it is the closure of the image of the compact mapping $J\colon E_K\rightarrow Y$, where $J$ is the natural embedding). By blowing up by a suitable factor, the closed unit ball $K_0$ of $L$ can be chosen such that $K$ is contained in $K_0$ by continuity. Finally, $K_0$ is compact in $E$ as reflexivity implies that $K_0$ is weakly compact in $L$ and therefore weakly compact in $E$ by continuity of the natural embedding. But this implies that $K_0$ is weakly closed and by convexity and precompactness in $E$ we have $K_0\in\nK_0(E)$.
\end{proof}
The following example establishes the connection of our approach to the work of \"Ust\"unel, cf.\ \cite{addprocessesnuclear}, who considered L\'evy processes with values in strong duals of nuclear spaces which are nuclear and Suslin. In these spaces $\nK_0^s$ is fundamental in $\nK_0$.
\begin{example}\label{ex:NuSCLS}
Let $E'$ be a nuclear Suslin space and a strong dual of a separable barreled nuclear space $E$. Then, if $K$ is a compact set in $E'$, there exists an absolutely convex compact set $S\supset K$ such that $E'_S$ is a separable Hilbert space. Thus, $\nK_0^s(E')$ is fundamental in $\nK_0(E')$.
\end{example}
\begin{proof}
In a nuclear space there exists a neighbourhood base $\nU$ such that for all $U\in \nU$ the completion of the space $E/p_U^{-1}(\{0\})$ is a separable Hilbert space $E_{(U)}$. Its dual space can be identified with $E'_{U^\circ}$, where $U^\circ$ is the polar of $U$. Define $\nK':=\{U^\circ \colon U\in \nU\}$ which is a fundamental system of closed bounded sets in $E'$ because $E$ is barreled, cf.\ \cite[5.2,~p.141]{Schaefer_TopologicalVectorSpaces}. As $E'$ is nuclear, all bounded sets are precompact, cf.\ \cite[p. 101, Corollary~2]{Schaefer_TopologicalVectorSpaces} and $\nK'$ consists of compact sets only. Choosing a compact set $K$ in $E'$ one finds an $S\in \nK'$ with $S=U^\circ$, $U\in \nU$, such that $K\subseteq S$. Consequently, $K\subseteq E'_S\cong (E_{(U)})'$ which is a separable Hilbert space.
\end{proof}
\begin{remark} \"Ust\"unel claims in his proof of Theorem III.1 \cite{addprocessesnuclear} that in his setting for a given $K\in\nK_0(E')$ one can always choose $S\in\nK'$ such that $K\subseteq S$. This means $\nK'$ is a fundemental system of bounded sets, which is only the case if $E$ is barreled. But this assumption is missing in the mentioned paper.
\end{remark}
One easily verifies the following stability properties:
\begin{lemma}
\begin{enumerate}
\item If $E_1,\ldots,E_n$ have fundamental systems of Banach disks, so does the locally convex direct sum $E_1\oplus\ldots \oplus E_n$.
\item If $E$ has a fundamental system of Banach disks, so does every closed subspace $F$.
\end{enumerate}
\end{lemma}
\subsection{A sufficient condition for $\nK_0^s(E)$ being fundamental}
The construction in Proposition~\ref{prop:bacls} is only known for Fr\'echet spaces. If for every compact disk $K$ in $E$ one can find a larger compact disk $B$ with compact embeddings $E_K\hookrightarrow E_{B} \hookrightarrow E$ one obtains a similar result. 
\begin{proposition}\label{prop:cpembed}
Let $K\in\nK_0(E)$. If there exists a compact disk $B\subseteq E$ containing $K$ and such that the canonical injection
$J\colon E_{K}\rightarrow E_B$ is compact, then there is also a compact disk $K_0\in\nK_0^s(E)$ containing $K$. 
\end{proposition}
\begin{proof}
First we note that $K$ is compact in $E_B$ as it is precompact by definition and closed by virtue of continuity of $E_B\hookrightarrow E$. The Banach space $E_B$ allows to find the desired compact set $K_0\in\nK_0^s(E_B)\subseteq \nK_0^s(E)$ by Proposition~\ref{prop:bacls}.
\end{proof}

In the literature, e.g.\ \cite{nuclearconuclear,jarchow}, the notion of a co-Schwartz space is well-established. Let $\mathfrak{S}$ be a system of bounded absolutely convex sets in $E$. A locally convex Hausdorff space $E$ is an \emph{$\mathfrak{S}$-co-Schwartz space} if for every $B\in \mathfrak S$ there exists $C\supseteq B$, $C\in\mathfrak S$ such that the natural embedding of the normed spaces $J_{BC}\colon E_B\rightarrow E_C$ admits a compact extension (to the completions). If $\mathfrak S$ is the space of closed disks, $E$ is called a co-Schwartz space. In the following, we choose $\mathfrak S=\nK_0(E)$. By Proposition~\ref{prop:cpembed} we obtain:

\begin{corollary}\label{cor:qcSCLS}
In $\nK_0$-co-Schwartz spaces the family $\nK_0^s(E)$ is fundamental.
\end{corollary}

\begin{corollary}
Let $K\in\nK_0(E)$. If there exists a bounded closed disk $B\subseteq E$ containing $K$ and such that the canonical injection $J\colon E_K\rightarrow E_B$ is compact, then there is also a compact disk $K_0\in\nK_0^s(E)$ containing $K$.
\end{corollary}
\begin{proof}
In view of Proposition~\ref{prop:bacls}, $K$ is compact in some Banach space $E_B$, $B\in\nB_0(E)$ and there is a compact disk $K_0\in \nK_0^s(E_B)\subseteq \nK_0^s(E)$ with $K_0\supseteq K$. 
\end{proof}
The property of $B$ being a Banach disk only depends on duality. In fact, the factorisation theorem in Proposition~\ref{prop:bacls} tells that the property of $E$ being a $\nK_0$-co-Schwartz space only depends on duality as the compact operator $\imath\colon  E_K\hookrightarrow E_B$ can be factored by two consecutive compact operators $E_K\hookrightarrow E_{K_0}\hookrightarrow E_B$ and $K$ is compact in $E_{K_0}$. This yields
\begin{corollary}\label{cor:compatibleqcSLS}
Let $(E,\t)$ be a locally convex space and $E'=(E,\t)'$. If on $E$ there is a $\nK_0$-co-Schwartz locally convex topology $\t'$ which is compatible with duality $\<E,E'\>$, then $(E,\t)$ is $\nK_0$-co-Schwartz. In this case, the family $\nK_0^s(E)$ is fundamental. Furthermore, then one can always choose compact sets from $\nK_0^s(E,\mu(E,E'))$, i.e., separable compact Banach disks in the Mackey topology.
\end{corollary}

\begin{corollary}\label{cor:frechetqcS}
Fr\'echet spaces are $\nK_0$-co-Schwartz.
\end{corollary}

\begin{proposition}\label{prop:qccS}
Co-Schwartz spaces are $\nK_0$-co-Schwartz.
\end{proposition}
\begin{proof}

If $E$ is a co-Schwartz space, it is quasi-complete, cf.\ \cite[Chapter~1, Theorem~(4d)]{nuclearconuclear}. Let $K\in\nK_0(E)$. It suffices to show that there exists $B\in\nK_0(E)$ such that the canonical embedding $J_{KB}$ is compact. For $K$ one finds a larger (not necessarily compact) disk $C$ such that the extension of the canonical embedding  $J_{KC}$ is compact. Without loss of generality $C$ is closed (e.g. take the closure of a suitable disk), so we assume that $C$ be closed, thus complete by quasi- completeness of $E$. Its linear hull $E_C$ is a Banach space. In particular, it is $\nK_0$-co-Schwartz by Corollary~\ref{cor:frechetqcS}. Therefore, one finds a compact disk $B\supseteq K$ in $E_C$ (and therefore in $E$)  and the assertion follows.
\end{proof}

\begin{remark}
Interestingly, although we need our assumptions for different purposes, Dettweiler posed essentially the same two conditions in \cite[Section~3]{dettweilerstabile}: In $E$ there should exist a fundamental system $\nK_H^s$ of $\nK_H$ of compact \emph{Hilbert} disks ($E_K$ is a separable Hilbert space for all $K\in\nK_H^s$). A second condition requests that for every $K\in\nK_H^s$ there is an $L\in\nK_H^s$, $K\subseteq L$, such that $\imath\colon E_K\rightarrow E_L$ is compact, i.e., it is a $\nK_H^s$-co-Schwartz space.
\end{remark}

\paragraph{Acknowledgements}
I want to express my deepest gratitude to my PhD supervisor Stefan Geiss for all the fruitful discussions and for his careful reading and valuable contributions to this work.

\def\polhk#1{\setbox0=\hbox{#1}{\ooalign{\hidewidth
  \lower1.5ex\hbox{`}\hidewidth\crcr\unhbox0}}}
  \def\cftil#1{\ifmmode\setbox7\hbox{$\accent"5E#1$}\else
  \setbox7\hbox{\accent"5E#1}\penalty 10000\relax\fi\raise 1\ht7
  \hbox{\lower1.15ex\hbox to 1\wd7{\hss\accent"7E\hss}}\penalty 10000
  \hskip-1\wd7\penalty 10000\box7}

\end{document}